\documentclass[10pt,a4paper]{amsart}
\pdfoutput=1
\usepackage{amsmath,amsfonts,amssymb,amsthm}
\usepackage{xcolor}
\usepackage{todonotes}
\usepackage[pagebackref=true]{hyperref}
\hypersetup{
            colorlinks   = true,
            pdfauthor    = {Giacomo Cherubini and Joao Guerreiro},
            pdftitle     = {Mean square in the prime geodesic theorem},
            pdfkeywords  = {},
            pdfcreator   = {Pdflatex},
            pdfpagemode  = UseNone,
            pdfstartview = FitH
           }
\usepackage{pdfpages}
\usepackage[normalem]{ulem}
\usepackage{url}
\usepackage{enumerate}
\usepackage{verbatim}
\usepackage{mathrsfs}
\usepackage{csquotes}
\usepackage{mathtools}
\usepackage{lipsum}
\numberwithin{equation}{section}

\newcommand\HH{\mathbb{H}}
\newcommand\NN{\mathbb{N}}

\newcommand\RR{\mathbb{R}}
\newcommand\ZZ{\mathbb{Z}}
\newcommand\eps{\varepsilon}
\newcommand\wh{\widehat}
\newcommand\GmodH{\Gamma\backslash\HH}

\DeclareMathOperator\PSL{PSL}
\DeclareMathOperator\SL{SL}
\DeclareMathOperator\Tr{Tr}
\DeclareMathOperator\vol{vol}
\DeclareMathOperator\li{li}

\theoremstyle{plain}
\newtheorem{thm}{Theorem}[section]

\newtheorem{lemma}[thm]{Lemma}

\newtheorem{prop}[thm]{Proposition}

\theoremstyle{definition}

\newtheorem{rmk}[thm]{Remark}

\title[Mean square in the prime geodesic theorem]{Mean square in the prime geodesic theorem}

\author{Giacomo Cherubini}
\address{Max-Planck-Institut f\"ur Mathematik, Vivatsgasse 7, 53111, Bonn, Germany}
\email{cherubini@mpim-bonn.mpg.de}

\author{Jo\~{a}o Guerreiro}
\address{Max-Planck-Institut f\"ur Mathematik, Vivatsgasse 7, 53111, Bonn, Germany}
\email{guerreiro@mpim-bonn.mpg.de}

\thanks{This work was supported by the Max-Planck-Institut f\"ur Mathematik. We would like to record our thanks to the Institute for the excellent working conditions.}
\keywords{Prime geodesic theorem, Selberg trace formula, Kuznetsov trace formula, Kloosterman sums}
\subjclass[2010]{Primary 11F72; Secondary 11M36, 11L05}
\date{\today}

\begin{document}

\begin{abstract}
We prove upper bounds for the mean square of the remainder in the prime geodesic theorem,
for every cofinite Fuchsian group, which improve on average on the best known pointwise bounds.
The proof relies on the Selberg trace formula.
For the modular group we prove a refined upper bound by using the Kuznetsov trace formula.
\end{abstract}

\maketitle

%%%%%%%%%%%%%%%%%%%%%%%%%%%%%%%%%%%%%%%%%%%%%%%%

\section{Introduction}

There is a striking similarity between the distribution of lengths of primitive closed geodesics
on the modular surface and prime numbers.
If we consider the asymptotic of the associated counting function,
the problem of finding optimal upper bounds on the remainder is intriguing,
especially because the relevant zeta function is known to satisfy the corresponding Riemann hypothesis.

We start with reviewing briefly the framework of the problem
(for a more detailed introduction we refer to \cite{sarnak_prime_1980}, \cite[\S 10.9]{iwaniec_spectral_2002}),
and then state our results.
Since the definitions extend to every finite volume Riemann surface,
we work in this generality, and only at a later point we specialize to the modular surface.

Every cofinite Fuchsian group $\Gamma$ acts on the hyperbolic plane $\HH$
by linear fractional transformations, and every hyperbolic element $g\in\Gamma$
is conjugated over $\SL_2(\RR)$ to a matrix of the form
\[\begin{pmatrix}\lambda^{1/2} &0 \\0&\lambda^{-1/2}\end{pmatrix},\quad\lambda\in\RR,\;\lambda>1.\]
The trace of $g$ is therefore $\Tr(g)=\lambda^{1/2}+\lambda^{-1/2}$,
and we define its norm to be $N(g)=\lambda$.
Since the trace and the norm are constant on the conjugacy class of~$g$,
if we set $P=\{\gamma g\gamma^{-1},\gamma\in\Gamma\}$,
we can define the trace and the norm of $P$ by $\Tr(P)=\Tr(g)$ and $NP=N(g)$.
Moreover, we say that $g$ (resp. $P$) is primitive
if $g$ cannot be expressed as a positive power (greater than one)
of another element $g_0\in\Gamma$.
Every class $P$ can then be written as $P=P_0^\nu$
for some primitive conjugacy class $P_0$ and some positive integer $\nu\geq 1$.

For $X>0$ define the counting function
\begin{equation}\label{intro:eq030}
\pi_\Gamma(X)  = \sharp\{P_0:NP_0\leq X\}.
\end{equation}

A hyperbolic conjugacy class $P$ determines a closed geodesic on the Riemann surface
$\GmodH$, and
its length is given exactly by $\log NP$.
For this reason we say that $\pi_\Gamma(X)$ in \eqref{intro:eq030} counts
primitive hyperbolic conjugacy classes in $\Gamma$, lengths of primitive geodesics on $\GmodH$,
or, by abusing language, \enquote{prime} geodesics.

It is a result from the middle of the last century
(see e.g. Huber \cite{huber_zur_1961,huber_zur_1961-1}, although the result was already known to Selberg)
that as $X\to\infty$ we have the asymptotic formula
\begin{equation}\label{1208:eq002}
\pi_\Gamma(X)\sim \li(X),
\end{equation}
where $\li(X)$ is the logarithmic integral.
The asymptotic coincides with that of the prime counting function,
since we have
\begin{equation}\label{intro:eq031}
\pi(X)=\sharp\{p\leq X: p\text{ prime}\}\sim \li(X)\quad\text{as }X\to\infty.
\end{equation}
The structure of the two problems is also similar. The study of primes
is strictly related to the study of the zeros of the Riemann zeta function $\zeta(s)$;
on the other hand, for primitive geodesics we need to study the zeros of
the Selberg zeta function $Z_\Gamma(s)$, which is defined for $\Re(s)>1$ by
\[
Z_\Gamma(s) = \prod_{P_0} \prod_{\nu=0}^\infty \big(1 - (NP_0)^{-\nu-s}\big),
\]
the outer product ranging over primitive hyperbolic classes in $\Gamma$,
and extends to a meromorphic function on the complex plane.

It is an interesting question to determine
the correct rate of the approximation in \eqref{1208:eq002} and \eqref{intro:eq031},
namely to prove optimal upper bounds on the differences
\begin{equation}\label{1208:eq001}
\pi(X)-\li(X)
\quad\text{and}\quad
\pi_\Gamma(X)-\li(X).
\end{equation}
In the case of primes, the conjectural estimate $|\pi(X)-\li(X)|\ll X^{1/2}\log X$
is equivalent to the Riemann hypothesis (see e.g. \cite[Th. 12.3]{ivic_theory_1985}).

To understand what type of control we should expect on the difference
on the right in \eqref{1208:eq001}, we introduce the weighted counting function
\[\psi_\Gamma(X) = \sum_{NP\leq X} \Lambda_\Gamma(NP),\]
where $\Lambda_\Gamma(NP)=\log(NP_0)$
if $P$ is a power of the primitive conjugacy class $P_0$, and $\Lambda_\Gamma(x)=0$ otherwise.
The function $\psi_\Gamma(X)$ is the analogous of the classical
summatory von Mangoldt function $\psi(X)$ in the theory of primes,
and studying $\pi_\Gamma(X)$ is equivalent to study $\psi_\Gamma(X)$,
but it is easier to work with the latter.

The asymptotic \eqref{1208:eq002} for $\pi_\Gamma(X)$ translates into $\psi_\Gamma(X)\sim X$
as $X$ tends to infinity.
The spectral theory of automorphic forms provides a finite number of additional secondary terms,
and we define the complete main term $M_\Gamma(X)$ to be the function
\begin{equation}\label{1208:eq003}
M_\Gamma(X) = \sum_{1/2<s_j\leq 1} \frac{\,X^{s_j}\!\!\!}{s_j} \, ,
\end{equation}
where $s_j$ is defined by $\lambda_j=s_j(1-s_j)$, $\lambda_j$ are the eigenvalues
of the Laplace operator~$\Delta$ acting on $L^2(\GmodH)$, and the sum is restricted
to the \emph{small eigenvalues}, that is, those satisfying $0\leq \lambda_j<1/4$.
Since the eigenvalues of $\Delta$ form a discrete set with no accumulation points,
the sum in \eqref{1208:eq003} is finite.
The remainder $P_\Gamma(X)$ is then defined as
\[
P_\Gamma(X) = \psi_\Gamma(X) - M_\Gamma(X).
\]

Various types of pointwise upper bounds have been proved
in the past years on $P_\Gamma(X)$, and we review them in Remark \ref{intro:rmk002} below.
In a different direction, it is possible to study the moments of $P_\Gamma(X)$,
and this has been neglected so far.

In the case of the prime number theorem, it is a classical result (assuming the Riemann hypothesis) that
the normalized remainder
\begin{equation*}
\frac{\psi(e^y)-e^y}{e^{y/2}}
\end{equation*}
admits moments of every order and a limiting distribution with
exponentially small tails
(see \cite[p. 242]{wintner_asymptotic_1935} and \cite[Th. 1.2]{rubinstein_chebyshevs_1994}).
The aim of this paper is to prove the following estimate on
the second moment of $P_\Gamma(X)$.

\begin{thm}\label{intro:thm01}
Let $\Gamma$ be a cofinite Fuchsian group. For $A\gg 1$ we have
\begin{equation*}
\frac{1}{A}\int_A^{2A} |P_\Gamma(X)|^2 dX \ll A^{4/3}.
\end{equation*}
\end{thm}

\begin{rmk}\label{intro:rmk001}
By construction the zeros of the Selberg zeta function correspond
to the eigenvalues of $\Delta$. The zeros corresponding to
small eigenvalues give rise to the secondary terms in the definition of $M_\Gamma(X)$,
and there are no other zeros in the half plane $\Re(s)>1/2$.
In this sense $Z_\Gamma(s)$ satisfies an analogous of the Riemann hypothesis,
and the fact suggests that we might have
\begin{equation}\label{1212:eq001}
P_\Gamma(X) \ll X^{1/2+\eps}\quad\forall\;\eps>0,
\end{equation}
since this is the case for the primes under RH.
However, the same method of proof fails, due to the abundance of zeros of $Z_\Gamma(s)$
(which in turn is related to the fact that $Z_\Gamma(s)$ is a function of order two,
while $\zeta(s)$ is of order one), and only gives
\begin{equation}\label{intro:V2:eq041}
P_\Gamma(X)\ll X^{3/4}.
\end{equation}
For a proof we refer to \cite[Th. 10.5]{iwaniec_spectral_2002}.
On the other hand, it is possible to prove that \eqref{1212:eq001}, if true,
is optimal, since we have the Omega result by Hejhal \cite[Th. 3.8, p. 477, and note 18, p. 503]{hejhal_selberg_1983}
\begin{equation*}
P_\Gamma(X) = \Omega_\pm(X^{1/2-\delta}) \quad\forall\;\delta>0,
\end{equation*}
which can be strengthened to $\Omega_\pm(X^{1/2}(\log\log X)^{1/2})$ for cocompact groups
and congruence groups.
\end{rmk}

\begin{rmk}
Recently Koyama \cite{koyama_refinement_2016}
and Avdispahi\'c \cite{avdispahic_koyamas_2017,avdispahic_prime_2017}
have tried an approach to find upper bounds on $P_\Gamma(X)$
using a lemma of Gallagher \cite[Lemma 1]{gallagher_large_1970}.
Their results are of the type
\[
P_\Gamma(X) \ll X^{\eta+\eps}\quad\forall\;\eps>0,
\]
where $\eta=7/10$ (resp. $\eta=3/4$) for $\Gamma$ cocompact (resp. ordinary cofinite),
and with $X$ outside a set $B\subseteq [1,\infty)$ of finite logarithmic measure,
that is, such that
\(
\int_B X^{-1} dX < \infty.
\)
Theorem \ref{intro:thm01} improves on this to the extent that
$\eta$ can be taken to be any number greater than $2/3$. Indeed,
a direct consequence of Theorem \ref{intro:thm01} is that
for every cofinite Fuchsian group, and for every $\eps>0$, the set
\begin{equation*}
B = \{ X\geq 1: |P_\Gamma(X)| \geq X^{2/3}(\log X)^{1+\eps}\}
\end{equation*}
has finite logarithmic measure.
Iwaniec claims without proof in \cite[p. 187]{iwaniec_nonholomorphic_1984}
that \enquote{it is easy to prove that $P_\Gamma(X)\ll X^{2/3}$ for almost all $X$}.
Theorem \ref{intro:thm01} proves his claim (up to $\eps$) for $X$ outside a set of finite logarithmic measure.
\end{rmk}

In the case of the modular group we can prove a stronger bound than
for the general cofinite case.
In this case we can exploit the
Kuznetsov trace formula,
obtaining the following.

\begin{thm}\label{intro:thm02}
Let $G=\PSL(2,\ZZ)$. For $A \gg 1$ and every $\eps>0$,
\begin{equation*}
\frac{1}{A} \int_A^{2A} \left| P_G(X) \right|^2 dX \ll A^{5/4+\varepsilon}.
\end{equation*}
\end{thm}

\begin{rmk}\label{intro:rmk002}
The estimate in \eqref{intro:V2:eq041}
is currently the best known pointwise upper bound
in the case of general cofinite Fuchsian groups.
For the modular group and congruence groups it is possible to
prove sharper estimates.
Iwaniec \cite[Th. 2]{iwaniec_prime_1984}, Luo and Sarnak \cite[Th. 1.4]{luo_quantum_1995},
Cai \cite[p. 62]{cai_prime_2002}, and Soundararajan and Young \cite[Th. 1.1]{soundararajan_prime_2013}
have worked on reducing the exponent for the modular group $G=\PSL(2,\ZZ)$.
The currently best known result is
\[
P_{G}(X) \ll X^{\eta+\eps},\quad\text{with}\quad \eta = \frac{25}{36}\quad\text{and for all}\quad\eps>0,
\]
due to Soundararajan and Young. The exponent $7/10$, proved by Luo and Sarnak,
holds also for congruence groups, see \cite[Cor. 1.2]{luo_selbergs_1995},
and for cocompact groups coming from quaternion algebras, see \cite{koyama_prime_1998}.

Observe that not only the exponent $\eta=1/2$ is out of reach,
but it seems to be a hard problem reaching and, afterwards, breaking the barrier $\eta=2/3$.
Iwaniec suggested (\cite[p. 188]{iwaniec_nonholomorphic_1984}, \cite[eq. (12)]{iwaniec_prime_1984}) that this follows from the assumption
of the Lindel\"of hypothesis for Rankin L-functions.
In Theorem \ref{intro:thm01} we prove the bound with the exponent $2/3$ on average, unconditionally, for every cofinite Fuchsian group.
Theorem \ref{intro:thm02} is saying that $P_G(X)\ll X^{5/8+\eps}$ on average. This goes halfway between the trivial bound \eqref{intro:V2:eq041}
and the conjectural estimate \eqref{1212:eq001}.
\end{rmk}

In the proof of Theorem \ref{intro:thm02} we use a truncated formula for $\psi_\Gamma(X)$,
proved by Iwaniec \cite[Lemma 1]{iwaniec_prime_1984},
condensing the analysis of the Selberg trace formula, and
relating $P_\Gamma(X)$ to the spectral exponential sum
\begin{equation*}
R(X,T) = \sum_{t_j\leq T} X^{it_j}.
\end{equation*}
The trivial bound for this sum, in view of Weyl's law
(see e.g. \cite[Cor. 11.2]{iwaniec_spectral_2002} or \cite[Th. 7.3]{venkov_spectral_1990}),
is $|R(X,T)|\ll T^2$.
Reducing the exponent of $T$ is possible at the cost of some power of $X$
(see \cite{iwaniec_prime_1984,luo_quantum_1995}),
and Petridis and Risager conjectured \cite[Conj. 2.2]{petridis_local_2014}
that we should have square root cancellation, that is,
\[R(X,T) \ll T^{1+\eps} X^\eps\]
for every $\eps>0$. This would give the conjectural bound \eqref{1212:eq001} for the modular group.
In the appendix to \cite{petridis_local_2014} Laaksonen provides numerics that support this conjecture.
In the proof of Theorem~\ref{intro:thm02}
(see Proposition~\ref{sharpsumbnd})
we prove that we have
\begin{equation*}
\frac{1}{A} \int_A^{2A} |R(X,T)|^2 dX \ll T^{2+\eps} A^{1/4+\eps},
\end{equation*}
and hence $|R(X,T)|\ll T^{1+\eps} X^{1/8+\eps}$ on average.

\medskip

The mean square estimate in Theorem \ref{intro:thm02} reduces
to proving a similar estimate for certain weighted sums of Kloosterman sums.
We obtain the desired bounds in a relatively clean way by applying
the Hardy-Littlewood-P\'olya inequality \cite[Th. 381, p. 288]{hardy_inequalities_1934}
in a special case: for $0<\lambda<1$, $\lambda=2(1-p^{-1})$, and $\{a_r\}$ a sequence of non-negative numbers,
then we have
\begin{equation}\label{HLP}
\sum_{r,s\atop r\neq s} \frac{a_r a_s}{|r-s|^\lambda} \ll_\lambda \Big(\sum_{r} a_r^p\Big)^{2/p}.
\end{equation}
With more work it is possible to sharpen Theorem \ref{intro:thm02}
by replacing $A^\eps$ by some power of $\log A$. To do this,
one can use a version of \eqref{HLP} with explicit implied constant
proved by Carneiro and Vaaler \cite[Cor. 7.2 eq. (7.20)]{carneiro_extremal_2010},
or the extremal case of \eqref{HLP} with $(\lambda,p)=(1,2)$,
and a logarithmic correction, proved by Li and Villavert \cite{li_extension_2011}.

\medskip

We also note that we simply use the Weil bound when estimating the weighted sums of Kloosterman sums and we do not exploit any cancellation amongst the Kloosterman sums. Exploring this phenomenon could lead to a power-saving improvement of Theorem~\ref{intro:thm02}.

\medskip

For the proof of Theorem \ref{intro:thm01} we use the Selberg trace formula with a suitably chosen test function that fullfills our needs.
In the general case of $\Gamma$ cofinite, unlike in the modular group case, Iwaniec's truncated formula for $\psi_\Gamma(X)$ is not available.
It is probably possible to prove an analogous result, but we preferred to work directly with the trace formula.
In fact, in order to prove his formula, Iwaniec uses some arithmetic information on the structure of the lengths of closed geodesics
on the modular surface \cite[Lemma 4]{iwaniec_prime_1984},
namely their connection with primitive binary quadratic forms (proved by Sarnak \cite[\S 1]{sarnak_class_1982}),
and it is unclear whether a similar argument works for the general cofinite case.

\begin{rmk}
We do not discuss the first moment of $P_\Gamma(X)$. 
Following the argument of \cite[Th. 1.1 and Th. 4.1]{phillips_circle_1994}
it should be possible to prove that
for every cofinite Fuchsian group $\Gamma$ there exists a constant $L_\Gamma$
such that we have
\begin{equation*}
\lim_{A\to\infty}
\frac{1}{\log A}\int_1^A \left(\frac{P_\Gamma(X)}{X^{1/2}}\right) \frac{dX}{X}
=
L_\Gamma.
\end{equation*}
\end{rmk}

%%%%%%%%%%%%%%%%%%%%%%%%%%%%%%%%%%%%%%%%%%%%%%%%%%%%%%%%%%%%%

\section{Proof of Theorem \ref{intro:thm01}}\label{section02}

In this section we explain the structure of the proof of Theorem \ref{intro:thm01}.
We decided to keep this section to a more colloquial tone,
and to relegate the technical computations to the next section.
Hence the proof of the theorem is completed in section~\ref{section03}.

The starting point of our analysis is the Selberg trace formula,
that we recall from the book of Iwaniec (see \cite[Th. 10.2]{iwaniec_spectral_2002}).
An even function~$g$ is said to be an admissible test function in the trace formula if
its Fourier transform $h(t)$ (see \eqref{intro:eq008} for the convention used
to define the Fourier transform)
satisfies the conditions (\cite[eq. (1.63)]{iwaniec_spectral_2002})
\begin{equation}\label{intro:eq001}
\begin{aligned}
& h(t) \text{ is even},\\
& h(t) \text{ is holomorphic in the strip } |\Im(t)|\leq 1/2+\eps,\\
& h(t) \ll (1+|t|)^{-2-\eps} \text{ in the strip.}
\end{aligned}
\end{equation}
For an admissible test function $g$, the Selberg trace formula is the identity
\begin{equation}\label{selberg-trace-formula}
\sum_{P} \; \frac{g(\log NP)}{2\sinh((\log NP)/2)} \; \Lambda_\Gamma(NP)
=
IE + EE + PE + DS + CS + AL,
\end{equation}
where the sum on the left runs over hyperbolic conjugacy classes in $\Gamma$,
and the terms appearing on the right are explained as follows.

The term $IE$ denotes a contribution coming from the identity element,
which forms a conjugacy class on its own. We have
\begin{equation*}
IE = -\frac{\vol(\GmodH)}{4\pi}\int_{-\infty}^{+\infty} th(t)\tanh(\pi t)dt.
\end{equation*}
The term $EE$ denotes a contribution from the elliptic conjugacy classes in $\Gamma$
(there are only finitely many such classes).
Denote by $\mathcal{R}$ a primitive elliptic conjugacy class, and let $m=m_\mathcal{R}>1$ be the order of $\mathcal{R}$.
We have
\begin{equation*}
EE = -\sum_\mathcal{R}\sum_{1\leq \ell<m} \big(2m\sin(\pi\ell/m)\big)^{-1}
\int_{-\infty}^{+\infty} h(t) \frac{\cosh(\pi(1-2\ell/m)t)}{\cosh(\pi t)} dt.
\end{equation*}
The term $PE$ denotes a contribution from the parabolic conjugacy classes,
which are associated to the cusps of $\Gamma$.
Let $\mathfrak{C}$ denote the total number of inequivalent cusps. Then we have
\begin{equation*}
PE = \frac{\mathfrak{C}}{2\pi}\int_{-\infty}^{+\infty} h(t) \psi(1+it)dt + \mathfrak{C}g(0)\log 2.
\end{equation*}
Here $\psi(z)$ is the digamma function, i.e. the logarithmic derivative
of the Gamma function, $\psi(z)=(\Gamma'/\Gamma)(z)$.
The term $DS$ corresponds to the discrete spectrum, and is given by
\begin{equation*}
DS = \sum_{t_j} h(t_j),
\end{equation*}
where the sum runs over the spectral parameter $t_j$, associated to the
eigenvalue $\lambda_j$ of the Laplace operator via the identity $\lambda_j=1/4+t_j^2$
(here and in the following we assume that $t_j\in [0,i/2]$ for $\lambda_j\in [0,1/4]$,
and $t_j>0$ for $\lambda_j>1/4$).
The term $CS$ comes from the continuous spectrum, and is given by
\begin{equation*}
CS = \frac{1}{4\pi}\int_{-\infty}^{+\infty} h(t) \frac{-\varphi'}{\varphi}(1/2+it)dt,
\end{equation*}
where $\varphi(s)$ is the scattering determinant for the group $\Gamma$
(see \cite[p. 140]{iwaniec_spectral_2002}).
Finally, the term $AL$ is a single term that comes from a combination
of the spectral part and the geometric part (see \cite[eq. (10.11) and (10.17)]{iwaniec_spectral_2002}),
and it is defined by
\(AL = h(0)\Tr(\Phi(1/2)-\mathrm{I})/4\),
where $\Phi(s)$ is the scattering matrix associated to $\Gamma$.

A first naive choice of test function in the Selberg trace formula \eqref{selberg-trace-formula} is
\begin{equation}\label{eq:g_s}
g_\rho(x)= 2\sinh\left(\frac{|x|}{2}\right)\mathbf{1}_{[0,\rho]}(|x|),
\end{equation}
where $\rho=\log X$. If we choose $g(x)$ in this way,
then the left hand side of \eqref{selberg-trace-formula} reduces exactly
to the function $\psi_\Gamma(X)$.
Unfortunately the function $g_\rho(x)$ in \eqref{eq:g_s}
is not an admissible test function, and we need therefore to take
a suitable approximation of it.
An analysis of the right hand side in \eqref{selberg-trace-formula}
then leads to the desired results.
To see that $g_\rho$ is not admissible in the trace formula,
consider its Fourier transform $h_\rho(t)$. We have
\begin{equation}\label{intro:eq008}
h_\rho(t)
=
\int_\RR g_\rho(x)e^{-itx}dx
=
2\int_0^\rho \sinh\left(\frac{x}{2}\right) (e^{itx}+e^{-itx})dx.
\end{equation}
The integral can be computed directly, which gives, for $t=\pm i/2$,
\begin{equation}\label{eq005}
h_\rho(\pm i/2) = 4\sinh^2(\rho/2),
\end{equation}
and, for $t\neq \pm i/2$,
\begin{equation}\label{eq004}
h_\rho(t)
=
\frac{2}{1/2+it}\cosh\left(\rho\big(1/2+it\big)\right)
+
\frac{2}{1/2-it}\cosh\left(\rho\big(1/2-it\big)\right)
-
\frac{2}{1/4+t^2}.
\end{equation}
The last integral in \eqref{intro:eq008} shows that $h_\rho(t)$
is even and entire in $t$, but from \eqref{eq004} we see that
we only have $h_\rho(t)\ll t^{-1}$ as $t$ tends to infinity,
and so we do not have sufficient decay as required in \eqref{intro:eq001}.

We construct two functions $g_\pm(x)$
that approximate from above and below the function $g_\rho$,
and are admissible in the Selberg trace formula.
Let $q(x)$ be an even, smooth, non-negative function on $\RR$
with compact support contained in $[-1,1]$ and unit mass (i.e. $\|q\|_1=1$).
Let $0<\delta<1/4$ and define
\begin{equation*}
q_\delta(x)=\frac{1}{\delta}q\left(\frac{x}{\delta}\right).
\end{equation*}
For $\rho>1$ define $g_\pm$ to be the convolution
product of the shifted function $g_{\rho\pm\delta}$ with the function $q_\delta$,
namely
\begin{equation}\label{2.13A}
g_\pm(x) = (g_{\rho\pm\delta}\ast q_\delta)(x)=\int_\RR g_{\rho\pm\delta}(x-y)q_\delta(y)dy.
\end{equation}
Taking convolution products has the advantage that the Fourier transform
of the convolution is the pointwise product of the Fourier transforms
of the two factors. Hence if we denote by
$\widehat{q}_\delta(t)$ the Fourier transform of $q_\delta$, that is,
\begin{equation}\label{3001:eq001}
\widehat{q}_\delta(t) = \int_\RR q_\delta(x) e^{-itx} dx,
\end{equation}
and by $h_\pm$ the Fourier transform of $g_\pm$,
then we obtain
\begin{equation}\label{2.14A}
h_\pm(t) = h_{\rho\pm\delta}(t)\wh{q}_\delta(t).
\end{equation}
Since the function $\wh{q}_\delta(t)$ is entire and satisfies, for $|\Im(t)|\leq M<\infty$,
\[
\wh{q}_\delta(t)\ll \frac{1}{1+\delta^k |t|^k}\quad\forall\;k\geq 0
\]
(see Lemma \ref{lemma001}), we conclude that $h_\pm(t)$ is an entire function
and satisfies $h_\pm(t)\ll (1+|t|)^{-2-\eps}$ for some $\eps>0$
in the strip $|\Im(t)|\leq 1/2+\eps$, as required in \eqref{intro:eq001}.
This shows that the function $g_\pm$ is an admissible test function in the trace formula.
By construction, the function $g_\pm(x)$ is supported on $|x|\in[0,\rho+\delta\pm\delta]$.
Moreover we have the inequalities (see Lemma \ref{lemma002}), for $x\geq 0$,
\begin{equation*}\label{intro:eq009}
g_{-}(x) + O(\delta e^{x/2}\mathbf{1}_{[0,\rho]}(x))
\leq
g_\rho(x)
\leq
g_{+}(x) + O(\delta e^{x/2}\mathbf{1}_{[0,\rho]}(x)).
\end{equation*}
If we set
\[
\psi_\pm(X)
=
\sum_{P} \; \frac{g_\pm(\log NP)}{2\sinh((\log NP)/2)} \; \Lambda_\Gamma(NP),
\]
then using the asymptotic $\psi_\Gamma(X)\sim X$
we conclude that we have the inequalities
\begin{equation*}
\psi_{-}(X) + O(\delta X) \leq \psi_\Gamma(X) \leq \psi_{+}(X) + O(\delta X).
\end{equation*}
From this we deduce that we have
\begin{equation}\label{outline:eq010}
\frac{1}{A}\int_A^{2A} |P(X)|^2 dX
\ll
\frac{1}{A}\int_A^{2A} |\psi_\pm(X)-M_\Gamma(X)|^2 dX + O(\delta^2 A^2),
\end{equation}
and in order to prove Theorem \ref{intro:thm01} we give bounds on the right
hand side in \eqref{outline:eq010}.
We found convenient to pass to the logarithmic variable $\rho=\log X$,
since this simplifies slightly the computations.
Moreover, we insert a weight function in the integral to pass from the sharp
mean square to a smooth one.
Consider a smooth, non-negative real function $w(\rho)$, compactly supported in $[-\eps,1+\eps]$,
for $0<\eps<1/4$. In addition, assume that $0\leq w(\rho)\leq 1$, and that $w(\rho)=1$ for $\rho\in[0,1]$.
Let
\begin{equation}\label{w_R}
w_R(\rho)=w(\rho-R). 
\end{equation}
In view of the inequality
\begin{equation}\label{outline:eq012}
\frac{1}{A}\int_A^{2A} |\psi_\pm(X)-M_\Gamma(X)|^2 dX
\ll
\int_\RR |\psi_\pm(e^\rho)-M_\Gamma(e^\rho)|^2 w_R(\rho)d\rho, 
\end{equation}
where $R=\log A$, we see that in order to estimate the last integral in \eqref{outline:eq010}
it suffices to upper bound the integral on the right in \eqref{outline:eq012}.
In the following, abusing notation, we write $\psi_\pm(\rho)$ in place of $\psi_\pm(e^\rho)$, 
and $M_\Gamma(\rho)$ in place of $M_\Gamma(e^\rho)$.

At this point we can exploit the Selberg trace formula to analyze $\psi_\pm$.
The terms of the discrete spectrum $DS$ associated to the
small eigenvalues $\lambda_j\in[0,1/4]$
(corresponding to the spectral parameter $t_j$ in the interval $[0,i/2]$,
and for technical convenience we include the eigenvalue $\lambda_j=1/4$)
need particular care.
We define $M_\pm(\rho)$ to be the sum of such terms, namely 
\[
M_\pm (\rho) := \sum_{t_j\in[0,i/2]} h_\pm(t_j). 
\]
In view of the definition of the complete main term in \eqref{1208:eq003}
it is easy to prove (see Lemma \ref{lemma003}) that we have
\[
M_\pm(\rho) = M_\Gamma(\rho) + O(\delta e^\rho + e^{\rho/2}). 
\]
Hence we have
\begin{align}
\psi_\pm(\rho)-M_\Gamma(\rho) 
&=
\psi_\pm(\rho)-M_\pm(\rho) + O(\delta e^\rho + e^{\rho/2}) 
\nonumber
\\
&=
IE_\pm + EE_\pm + PE_\pm + DS_\pm' + CS_\pm + AL_\pm + O(\delta e^\rho + e^{\rho/2}),\label{outline:eq013} 
\end{align}
where the term $DS_\pm'$ denotes now the contribution from the discrete spectrum
restricted to the eigenvalues $\lambda_j>1/4$.
The first three terms  and the term $AL_\pm$ in \eqref{outline:eq013}
can be bounded pointwise
(see Lemma \ref{lemma004}\,--\,\ref{lemma09})
by
\[
|IE_\pm| + |EE_\pm| + |PE_\pm| + |AL_\pm|
\ll
e^{\rho/2} + \log(\delta^{-1}). 
\]
Squaring and integrating in \eqref{outline:eq013} we obtain therefore
\begin{equation}\label{outline:eq014}
\begin{split}
\int_\RR |\psi_\pm(\rho)-M_\Gamma(\rho)|^2 w_R(\rho)d\rho 
&\ll
\int_\RR |DS_\pm'|^2 w_R(\rho) d\rho + \int_\RR |CS_\pm|^2 w_R(\rho) d\rho 
\\
&+
O\big(
e^R + \delta^2 e^{2R} + \log^2(\delta^{-1})
\big).
\end{split}
\end{equation}
In Proposition \ref{prop01} and Proposition \ref{prop02} we prove that the following estimate holds:
\begin{equation}\label{outline:eq015}
\int_\RR |DS_\pm'|^2 w_R(\rho) d\rho 
+
\int_\RR |CS_\pm|^2 w_R(\rho) d\rho 
\ll
\frac{e^R}{\delta} + e^{R/2}\log^2(\delta^{-1}).
\end{equation}
Combining \eqref{outline:eq014} and \eqref{outline:eq015} we obtain
\[
\int_\RR |\psi_\pm(\rho)-M_\Gamma(\rho)|^2 w_R(\rho)d\rho 
\ll
\frac{e^R}{\delta} + e^{R/2}\log^2(\delta^{-1}) + \delta^2 e^{2R},
\]
and choosing $\delta=e^{-R/3}$ to optimize the first and the last term,
we arrive at the bound
\[
\int_\RR |\psi_\pm(s)-M_\Gamma(s)|^2 w_R(\rho)d\rho 
\ll
e^{4R/3}.
\]
Recalling \eqref{outline:eq010} and \eqref{outline:eq012},
and setting $R=\log A$, we conclude that we have
\[
\frac{1}{A}\int_A^{2A} |P(X)|^2 dX
\ll
\int_\RR |\psi_\pm(\rho)-M_\Gamma(\rho)|^2 w_R(\rho)d\rho + O(\delta^2 A^2) 
\ll
A^{4/3}.
\]
This proves Theorem \ref{intro:thm01}.

%%%%%%%%%%%%%%%%%%%%%%%%%%%%%%%%%%%%%%%%%%%%%%%%%%%%%%%%%%%%%%%

%%%%%%%%%%%%%%%%%%%%%%%%%%%%%%%%%%%%%%%%%%%%%
%%%%%%%    TECHNICAL  LEMMATA      %%%%%%%%%%

\section{Technical lemmata}\label{section03}

%%%%%%%%%%%%%%%%%%%%%%%%%%%%%%%%%%%%%%%%%%%%%

In this section we prove the auxiliary results needed in section \ref{section02}
to prove Theorem \ref{intro:thm01}.
We start with a simple computation to bound the function $\widehat{q}_\delta(t)$ and its derivatives.

%%%%%%%%%%%%%%%%%%%%%%%%%%%%%%%%%%%%%%%%%%%%%%%%%
%%%  LEMMA 1: DERIVATIVES OF \hat{q}_\delta  %%%%

\begin{lemma}\label{lemma001}
Let $0<\delta<1/4$, $\widehat{q}_\delta(t)$ as in \eqref{3001:eq001}, and let $t\in\RR$.
Let $j,k\in\NN$, $j,k\geq 0$. We have
\begin{equation*}
\Big|\frac{d^j\widehat{q}_\delta}{dt^j}(t)\Big|
\ll
\frac{\delta^j}{1+|\delta t|^k},
\end{equation*}
and the implied constant depends on $j$ and $k$.
\end{lemma}
%%%%%%%%%%%%%%%%%%%%%%%%%%%%%%%%%%%%%%%%
%%%%%     PROOF OF LEMMA 1      %%%%%%%%
\begin{proof}
From the definition of $\wh{q}_\delta(t)$ we have
\begin{equation}\label{eq001}
\frac{d^j\widehat{q}_\delta}{dt^j}(t)
=
\frac{d^j}{dt^j} \int_\RR q(x) e^{-it\delta x} dx
=
(-i\delta)^j\int_\RR x^j q(x) e^{-it\delta x}dx.
\end{equation}
Bounding in absolute value we get
\begin{equation}\label{eq002}
\frac{d^j\widehat{q}_\delta}{dt^j}(t) \ll \delta^j.
\end{equation}
Integrating by parts in the last integral in \eqref{eq001}, and using that $q(x)$
is smooth and has compact support, we obtain instead
\begin{equation}\label{eq003}
\frac{d^j\widehat{q}_\delta}{dt^j}(t)
=
(-1)^k \frac{(-i\delta)^j}{(-it\delta)^k} \int_\RR e^{-it\delta x} \frac{d^k}{dx^k}(x^j q(x)) dx
\ll
\frac{\delta^j}{|\delta t|^k}.
\end{equation}
Combining \eqref{eq002} and \eqref{eq003} we obtain
\begin{equation*}
\frac{d^j\widehat{q}_\delta}{dt^j}(t)
\ll
\delta^j\min(1,|\delta t|^{-k})\ll\frac{\delta^j}{1+|\delta t|^k}.
\end{equation*}
This proves the lemma.
\end{proof}
%%%%%%      END OF PROOF OF LEMMA 1      %%%%%%%%
%%%%%%%%%%%%%%%%%%%%%%%%%%%%%%%%%%%%%%%%%%%%%%%%%

The function $\widehat{q}_\delta(t)$ is used to construct the convolution product
$g_\pm=g_{\rho\pm\delta}\ast q_\delta$ that approximate the function $g_\rho$. 
In the next lemma we show that $g_\rho$ is bounded from above and below by $g_{+}$ and $g_{-}$ 
respectively, up to a small error.

%%%%%%%%%%%%%%%%%%%%%%%%%%%%%%%%%%%%%%%%%%%%%%%%%%
%%% LEMMA 2: INEQUALITY g_{-} <= g_s <= g_{+}  %%%
\begin{lemma}\label{lemma002}
Let $\rho>1$ and $0<\delta<1/4$. 
Let $g_\rho(x)$ and $g_\pm(x)$ as in \eqref{eq:g_s} and \eqref{2.13A} respectively. 
For $x\geq 0$ we have
\begin{equation}\label{eq:032}
g_{-}(x) + O(\delta e^{x/2}\mathbf{1}_{[0,\rho]}(x)) 
\leq
g_\rho(x) 
\leq
g_{+}(x) + O(\delta e^{x/2}\mathbf{1}_{[0,\rho]}(x)). 
\end{equation}
\end{lemma}
%%%%%%%%%%%%%%%%%%%%%
%% PROOF OF LEMMA  %%
\begin{proof}
Observe that $g_\pm = g_\rho\ast q_\delta$, and since both factors in the convolution product are 
non-negative, we conclude that $g_\pm(x)\geq 0$ for every $x\in\RR$.
Observe also that $g_\pm$ is supported on the compact set $\{|x|\leq \rho+\delta\pm\delta\}$. 
By definition of $g_\pm$ we can therefore write, for $x\geq 0$,
\[
\begin{aligned}
g_\pm(x)
&=
2\int_\RR \mathbf{1}_{[0,\rho\pm\delta]}(|x-y|)\sinh\left(\frac{|x-y|}{2}\right)q_\delta(y)dy 
\\
&=
2\cdot\mathbf{1}_{[0,\rho+\delta\pm\delta]}(x) \int_Q \sinh\left(\frac{|x-y|}{2}\right)q_\delta(y)dy, 
\end{aligned}
\]
where $Q=[\max\{-\delta,x-(\rho\pm\delta)\},\delta]$. 
For $0\leq x\leq \delta$ we have
\[g_\rho(x)=O(\delta)\quad\text{and}\quad g_\pm(x)=O(\delta),\] 
so that \eqref{eq:032} holds trivially.

For $x>\delta$ we have $x-y>0$ for every $y\in Q$, and using the addition formula
for the hyperbolic sine we can write
\begin{align}
g_\pm(x)
&=
2\cdot\mathbf{1}_{[0,\rho+\delta\pm\delta]}(x) \int_Q \big[\cosh(y/2)\sinh(x/2)-\cosh(x/2)\sinh(y/2)\big]q_\delta(y)dy 
\nonumber
\\
&=
2\cdot\mathbf{1}_{[0,\rho+\delta\pm\delta]}(x) \sinh\left(\frac{x}{2}\right) 
\int_Q q_\delta(y)dy
+
O(\delta e^{x/2}\mathbf{1}_{[0,\rho+\delta\pm\delta]}(x)). 
\label{eq:033}
\end{align}
Now for $\delta<x<\rho\pm\delta-\delta$ we have $Q=[-\delta,\delta]$, so that 
\eqref{eq:033} reduces to
\[g_\pm(x)= g_\rho(x) + O(\delta e^{x/2}\mathbf{1}_{[0,\rho+\delta\pm\delta]}(x)).\] 
For $\rho-2\delta\leq x\leq \rho$ we can bound by positivity in \eqref{eq:033} 
\[g_{-}(x)\leq g_\rho(x)+O(\delta e^{x/2}\mathbf{1}_{[0,\rho]}(x)),\] 
so that the first inequality in \eqref{eq:032} holds in this case.
Finally for $\rho\leq x\leq \rho+2\delta$ we observe that $g_\rho(x)=0$ and $g_{+}(x)\geq 0$, 
and so we conclude that the second inequality in \eqref{eq:032}
holds in this case.
This proves the lemma.
\end{proof}
%%%  END OF PROOF INEQUALITY g_\pm %%%
%%%%%%%%%%%%%%%%%%%%%%%%%%%%%%%%%%%%%%

The function $h_{\pm}(t)$ associated to $g_\pm(x)$ gives, for $t$ corresponding
to small eigenvalues, the terms appearing in the definition of the main term $M_\Gamma(X)$ in \eqref{1208:eq003},
up to small error. We prove this in the next lemma.

%%%%%%%%%%%%%%%%%%%%%%%%%%%%%%%%%%%%%%%%%%%%%%%%%
%%%%  LEMMA 3: SMALL EIGENVALUES FOR h_\pm   %%%%
\begin{lemma}\label{lemma003}
Let $\rho>1$, $0<\delta<1/4$, and let $h_\pm(t)$ as in \eqref{2.14A}. 
Let $t_j\in(0,i/2]$ be the
spectral parameter associated to the eigenvalue $\lambda_j\in[0,1/4)$.
There exists $0<\eps_\Gamma<1/4$ such that
\begin{equation}\label{lemma003:equation}
h_\pm(t_j)= \frac{e^{\rho(1/2+|t_j|)}}{1/2+|t_j|} + O(\delta e^\rho + e^{\rho(1/2-\eps_\Gamma)}). 
\end{equation}
\end{lemma}
%%%%%%%%%%%%%%%%%%%%%%%%%%%%%%%%%%%%%%%%
%%%%%%   PROOF OF LEMMA 3      %%%%%%%%%
\begin{proof}
Recall that there are only finitely many eigenvalues $\lambda_j$ in $(0,1/4)$,
and therefore there exists $0<\eps_\Gamma<1/4$ such that $|t_j|\geq \eps_\Gamma$
for every $\lambda_j\in(0,1/4)$.
Observe that for $|t|\leq 1/2$ we have
\[\widehat{q}_\delta(t)=\int_\RR q(x)e^{-it\delta x}dx = 1+ O(\delta|t|).\]
The claim then follows from \eqref{eq005} and \eqref{eq004},
and from the fact that $h_\pm(t)=h_{\rho\pm\delta}(t)\widehat{q}_\delta(t)$. 
\end{proof}
\begin{rmk}
Equation \eqref{lemma003:equation} also holds for $\lambda_j=1/4$, except that
we have to multiply the first term by a factor of~$2$.
\end{rmk}
%%%%%%   END OF PROOF OF LEMMA 3    %%%%%%%%
%%%%%%%%%%%%%%%%%%%%%%%%%%%%%%%%%%%%%%%%%%%%

The next three lemmata show that in the problem of estimating $P_\Gamma(X)$ using the
Selberg trace formula we can neglect the terms coming from the identity class,
the elliptic classes, and the parabolic classes, as they contribute a small quantity.

%%%%%%%%%%%%%%%%%%%%%%%%%%%%%%%%%%%%%%%%%%%%%%%%%%
%%%%%   LEMMA 4: IDENTITY CONTRIBUTION    %%%%%%%%
\begin{lemma}\label{lemma004}
Let $\rho>1$, $0<\delta<1/4$, and let 
$g_\pm$ and $h_\pm$ as in \eqref{2.13A} and \eqref{2.14A} respectively.
Then we have
\begin{equation*}
\int_{-\infty}^{+\infty} t h_\pm(t) \tanh(\pi t) dt
\ll
\frac{e^{\rho/2}}{\rho^2} + \log(\delta^{-1}). 
\end{equation*}
\end{lemma}
%%%%%%%%%%%%%%%%%%%%%%%%%%%%%%%%%%%%%%%%
%%%%%%   PROOF OF LEMMA 4      %%%%%%%%%
\begin{proof}
Recall that we have $h_\pm(t)=h_\rho(t)\widehat{q}_\delta(t)$, 
and that by Lemma \ref{lemma001} we have $\widehat{q}_\delta(t)\ll (1+|\delta t|^k)^{-1}$
for every $k\geq 0$.
Using \eqref{eq004} and the definition of the hyperbolic cosine we can write
\begin{equation}\label{eq009a}
\begin{split}
h_{\rho\pm\delta}(t) 
&=
\frac{2}{1+2it}(e^{(\rho\pm\delta)(1/2+it)}+e^{-(\rho\pm\delta)(1/2+it)}) 
\\
&+
\frac{2}{1-2it}(e^{(\rho\pm\delta)(1/2-it)}+e^{-(\rho\pm\delta)(1/2-it)}) 
-
\frac{2}{1/4+t^2}.
\end{split}
\end{equation}
Bounding in absolute value the integrand associated to the
last term in \eqref{eq009a}, we obtain
\begin{equation}\label{eq013}
\int_\RR \Big|\frac{t\,\widehat{q}_\delta(t)\tanh(\pi t)}{1/4+t^2}\Big|dt
\ll
\int_\RR \frac{dt}{(1+|t|)(1+|\delta t|)}
\ll
\log(\delta^{-1})+1.
\end{equation}
Consider now the integrand associated to the first term in \eqref{eq009a}
We integrate by parts twice and obtain
\begin{align*}
I:= 2\int_\RR
&\frac{t\,\widehat{q}_\delta(t) \tanh(\pi t)e^{(\rho\pm\delta)(1/2+it)}}{1+2it}dt 
\\
&=
2\sum_{j=1}^2 (-1)^{j-1}\frac{e^{(\rho\pm\delta)(1/2+it)}}{(i(\rho\pm\delta))^j} 
\frac{d^{j-1}}{dt^{j-1}}\left(\frac{t\,\widehat{q}_\delta(t)\tanh(\pi t)}{(1+2it)}\right)\bigg|_{t=-\infty}^{+\infty}
\\
&+
2\int_\RR \frac{e^{(\rho\pm\delta)(1/2+it)}}{(i(\rho\pm\delta))^2} 
\frac{d^2}{dt^2}\left(\frac{t\,\widehat{q}_\delta(t)\tanh(\pi t)}{(1+2it)}\right)dt.\rule{0pt}{17pt} 
\end{align*}
By Lemma \ref{lemma001} we see that the boundary terms vanish, and we can bound
\[
\frac{d^2}{dt^2}\left(\frac{t\,\widehat{q}_\delta(t)\tanh(\pi t)}{(1+2it)}\right)
\ll
\frac{1}{(1+|\delta t|^k)}\left(\delta^2+\frac{1}{1+|t|^2}\right),
\]
for every $k\geq 0$, with implied constant depending on $k$.
Fixing $k>1$ we obtain
\begin{equation}\label{eq014}
I
\ll
\frac{e^{\rho/2}}{\rho^2}\int_\RR \frac{1}{(1+|\delta t|^k)}\left(\delta^2+\frac{1}{1+|t|^2}\right) dt 
\ll
\frac{e^{\rho/2}}{\rho^2}. 
\end{equation}
The other terms in \eqref{eq009a}
are treated similarly, and are bounded by the same quantity.
Adding \eqref{eq013} and \eqref{eq014} we obtain the desired estimate.
\end{proof}
%%%%%%%     END OF PROOF OF LEMMA 4    %%%%%%%%
%%%%%%%%%%%%%%%%%%%%%%%%%%%%%%%%%%%%%%%%%%%%%%%

%%%%%%%%%%%%%%%%%%%%%%%%%%%%%%%%%%%%%%
%%%%  LEMMA: ELLIPTIC ELEMENTS   %%%%%
\begin{lemma}\label{lemma08}
Let $\rho>1$, $0<\delta<1/4$, and $h_\pm$ as in \eqref{2.14A}. 
Let $m,\ell\in\NN$, with $m\geq 2$ and $1\leq \ell<m$. Then 
\[
\int_\RR h_\pm(t) \frac{\cosh(\pi(1-2\ell/m)t)}{\cosh(\pi t)}dt
\ll
\frac{e^{\rho/2}}{\rho^2}, 
\]
with implied constant that depends on $m$ and $\ell$.
\end{lemma}
%%%%%%%%%%%%%%%%%%%%%%
%%  PROOF OF LEMMA  %%
\begin{proof}
Recall that $h_\pm(t)=h_{\rho\pm\delta}(t)\wh{q}_\delta(t)$, and use \eqref{eq009a} 
to express $h_{\rho\pm\delta}(t)$. The integrand associated to the last term in \eqref{eq009a} 
is bounded by
\[
\ll
\int_\RR \frac{\cosh(\pi(1-2\ell/m)t)}{(1+t^2)\cosh(\pi t)}dt
\ll
\int_\RR \frac{dt}{1+t^2} \ll 1.
\]
Consider now the term $\exp((\rho\pm\delta)(1/2+it))$ in \eqref{eq009a}. 
The corresponding integral contributes
\[
\begin{aligned}
J:=
2\int_\RR
&\frac{e^{(\rho\pm\delta)(1/2+it)}\wh{q}_\delta(t)\cosh(\pi(1-2\ell/m)t)}{(1+2it)\cosh(\pi t)}dt 
\\
&=
2\sum_{j=1}^2 (-1)^{j-1}
\frac{e^{(\rho\pm\delta)(1/2+it)}}{(i(\rho\pm\delta))^j} 
\frac{d^{j-1}}{dt^{j-1}}
\left(
\frac{\wh{q}_\delta(t)\cosh(\pi t(1-2\ell/m))}{(1+2it)\cosh(\pi t)}
\right)\bigg|_{t=-\infty}^{+\infty}
\\
&+ 2\int_\RR
\frac{e^{(\rho\pm\delta)(1/2+it)}}{(i(\rho\pm\delta))^2} 
\frac{d^2}{dt^2}
\left(
\frac{\wh{q}_\delta(t)\cosh(\pi t(1-2\ell/m))}{(1+2it)\cosh(\pi t)}
\right)dt.
\end{aligned}
\]
The boundary terms vanish, and by Lemma \ref{lemma001} we can bound
\[
\frac{d^2}{dt^2}
\left(
\frac{\wh{q}_\delta(t)\cosh(\pi t(1-2\ell/m))}{(1+2it)\cosh(\pi t)}
\right)
\ll
\frac{1}{(1+|t|)(1+|\delta t|^k)}\left(\frac{1}{1+|t|^2}+\delta^2\right)
\]
where the implied constant depends on $m$ and $\ell$.
Hence we get
\[
J
\ll
\frac{e^{\rho/2}}{\rho^2}\int_\RR \frac{1}{(1+|t|)(1+|\delta t|^k)}\left(\frac{1}{1+|t|^2}+\delta^2\right) dt 
\ll
\frac{e^{\rho/2}}{\rho^2}. 
\]
The terms associated to the other exponentials in \eqref{eq009a} are treated similarly,
and are bounded by the same quantity.
\end{proof}
%%%%  END OF PROOF LEMMA ELLIPTIC ELEMENTS %%%%
%%%%%%%%%%%%%%%%%%%%%%%%%%%%%%%%%%%%%%%%%%%%%%%

%%%%%%%%%%%%%%%%%%%%%%%%%%%%%%%%%%%%%
%%%%  LEMMA PARABOLIC ELEMENTS   %%%%
\begin{lemma}\label{lemma09}
Let $\rho>1$, $0<\delta<1/4$, and let $g_\pm$ and $h_\pm$ 
as in \eqref{2.13A} and \eqref{2.14A}. Then
\[
g_\pm(0)\log 2
+
\frac{1}{2\pi}\int_{-\infty}^{+\infty} h_\pm(t) \psi(1+it)dt
\ll
\delta e^{\rho/2} + \log(\delta^{-1}). 
\]
\end{lemma}
%%%%%%%%%%%%%%%%%%%%%%%%%
%%%  PROOF OF LEMMA  %%%%
\begin{proof}
We use here the formula given in \cite[eq. (10.17)]{iwaniec_spectral_2002}
to get back to an integral involving $g_\pm$. We have
\begin{equation}\label{eq017}
\begin{split}
g_\pm(0)\log 2
&+
\frac{1}{2\pi}\int_{-\infty}^{+\infty} h_\pm(t) \psi(1+it)dt
\\
&=
\frac{h_\pm(0)}{4}
-
\gamma g_\pm(0)
+
\int_0^\infty \log\big(\sinh(x/2)\big) dg_\pm(x).
\end{split}
\end{equation}
By \eqref{eq009a} we know that
\begin{equation}\label{eq018}
\frac{h_\pm(0)}{4}=e^{\rho/2}+O(\delta e^{\rho/2}+1), 
\end{equation}
and by definition of $g_\pm(x)$ we have $g_\pm(x)=O(\delta)$ for $|x|\leq \delta$.
The last integral in \eqref{eq017} is analyzed as follows.
For $x\in[0,\delta]$ we bound
\[
\frac{d}{dx} g_\pm(x)
=
\big(g_{\rho\pm\delta}\ast q_\delta'\big)(x) 
\ll
\int_{-\delta}^{\delta} \sinh\Big(\frac{|x-y|}{2}\Big)|q_\delta'(y)|dy
\ll
\frac{\sinh(\delta)}{\delta} \|q'\|_1\ll 1,
\]
so that we have, uniformly in $\delta$,
\[
\int_0^\delta \log(\sinh(x/2)) dg_\pm(x)
\ll
\int_0^\delta |\log(\sinh(x/2))| dx \ll 1.
\]
For $x>\delta$ we integrate by parts obtaining
\[
\begin{split}
\int_\delta^\infty \log(\sinh(x/2))dg_\pm(x)
&=
-g_\pm(\delta)\log(\sinh(\delta/2))
\\
&-
\frac{1}{2}\int_\delta^\infty \frac{g_\pm(x)}{\sinh(x/2)}\cosh(x/2)dx.
\end{split}
\]
Since $g_\pm(\delta)\ll\delta$, the boundary term is bounded by $O(\delta\log(\delta^{-1}))$.
If we write $\cosh(x/2)=\sinh(x/2)+e^{-x/2}$, the integral associated to $e^{-x/2}$ can be bounded by
\[
\int_\delta^\infty \frac{g_\pm(x)e^{-x/2}}{\sinh(x/2)}dx
\ll
\int_\delta^1 \frac{dx}{x} + \int_1^\infty e^{-x/2}dx \ll \log(\delta^{-1})+1.
\]
Finally we have
\[
-\frac{1}{2}\int_\delta^\infty g_\pm(x)dx
=
-2\cosh\Big(\frac{\rho\pm\delta}{2}\Big) + O(1) 
=
-e^{\rho/2} + O(\delta e^{\rho/2}+1). 
\]
The exponential cancels with the first term in \eqref{eq018},
and combining the other estimates we obtain the claim.
\end{proof}
%%%% END OF PROOF PARABOLIC ELEMENTS  %%%%
%%%%%%%%%%%%%%%%%%%%%%%%%%%%%%%%%%%%%%%%%%

We turn now our attention to finding upper bounds for the mean square of the spectral side
in the Selberg trace formula. We start with an estimate for the integral of $h_\pm(t_1)h_\pm(t_2)$.

%%%%%%%%%%%%%%%%%%%%%%%%%%%%%%%%%%%%%%%%%%%%%%%%%%%%%%%%%%
%%%%%   LEMMA : INTEGRAL OF TWO  h_{s\pm\delta}    %%%%%%%
\begin{lemma}\label{lemma005}
Let $0<\delta<1/4$, let $h_{\rho\pm\delta}(t)$ be as in \eqref{2.14A}, 
and let $w_R(\rho)$ be as in \eqref{w_R}. 
Let $R>1$, and $t_1,t_2\in\RR$. Then we have
\begin{equation*}\label{eq012}
\begin{split}
\int_\RR h_{\rho\pm\delta}(t_1)h_{\rho\pm\delta}(t_2) w_R(\rho)d\rho 
\ll
\frac{e^Rv(t_1)v(t_2)}{1+|t_1-t_2|^2}
&+
\frac{e^Rv(t_1)v(t_2)}{1+|t_1+t_2|^2}
\\
&+
e^{R/2}v(t_1^2)v(t_2^2),\rule{0pt}{13pt}
\end{split}
\end{equation*}
where $v(t)=(1+|t|)^{-1}$, and the implied constant does not depend on $\delta$.
\end{lemma}
%%%%%%%%%%%%%%%%%%%%%%%%%%%%%%%%%%%%%%%%
%%%%%%   PROOF OF LEMMA 4      %%%%%%%%%
\begin{proof}
In order to express $h_{\rho\pm\delta}$ we consider again equation \eqref{eq009a}. 

Multiplying $h_{\rho\pm\delta}(t_1)$ with $h_{\rho\pm\delta}(t_2)$, the product of 
the first exponential in \eqref{eq009a} (from the factor $h_{\rho\pm\delta}(t_1)$) 
and the last term from the factor $h_{\rho\pm\delta}(t_2)$ contributes 
\begin{align}
\int_\RR \frac{4e^{(\rho\pm\delta)(1/2+it_1)}}{(1+2it_1)(1/4+t_2^2)} w_R(\rho)d\rho 
&=
\int_\RR \frac{2e^{(\rho\pm\delta)(1/2+it_1)}w_R'(\rho)}{(1/2+it_1)^2(1/4+t_2^2)} d\rho 
\nonumber
\\
&\ll
e^{R/2} v(t_1^2)v(t_2^2) \|w'\|_1
\ll
e^{R/2} v(t_1^2)v(t_2^2).\rule{0pt}{14pt}
\label{eq010}
\end{align}
Consider now the product of the first exponential in \eqref{eq009a} for $h_{\rho\pm\delta}(t_1)$ 
and the same term for $h_{\rho\pm\delta}(t_2)$. This contributes 
\begin{align}
\int_\RR \frac{4e^{(\rho\pm\delta)(1+i(t_1+t_2))}}{(1+2it_1)(1+2it_2)} w_R(\rho)d\rho 
&\ll
e^Rv(t_1)v(t_2)\min(1,|t_1+t_2|^{-2}\|w''\|_1)
\nonumber
\\
&\ll
\frac{e^R v(t_1)v(t_2)}{1+|t_1+t_2|^2}.
\label{eq011}
\end{align}
The other terms in the product $h_{\rho\pm\delta}(t_1)h_{\rho\pm\delta}(t_2)$ are bounded similarly by 
\eqref{eq010} and \eqref{eq011}, except that we need to replace $|t_1+t_2|$
by $|t_1-t_2|$ when we integrate the product $\exp((\rho\pm\delta)(\pm 1\pm i(t_1-t_2))$. 
This concludes the proof.
\end{proof}
%%%%%%   END OF PROOF OF LEMMA 4    %%%%%%%%
%%%%%%%%%%%%%%%%%%%%%%%%%%%%%%%%%%%%%%%%%%%%

In order to exploit at best the bound proved in the previous lemma, we 
estimate the size of the spectrum on unit intervals.

%%%%%%%%%%%%%%%%%%%%%%%%%%%%%%%%%%%%%%%%%%%%%%%%%
%%%%  LEMMA: SHORT INTERVAL WEYL INEQUALITY  %%%%
\begin{lemma}\label{lemma006}
Let $\Gamma$ be a cofinite Fuchsian group,
let $\varphi(s)$ be the scattering determinant associated to $\Gamma$,
and let $T>1$. We have
\[
\sharp\{T\leq t_j\leq T+1\} + \int_{T\leq|t|\leq T+1} \Big|\frac{-\varphi'}{\varphi}(1/2+it)\Big|dt \ll T.
\]
\end{lemma}
%%%%%%%%%%%%%%%%%%%%%%%%%%%%
%%%  PROOF OF THE LEMMA  %%%
\begin{proof}
Recall Weyl's law in its strong form (see \cite[Th. 7.3]{venkov_spectral_1990})
\begin{equation*}
\sharp\{t_j\leq T\} + \frac{1}{4\pi}\int_{-T}^T \frac{-\varphi'}{\varphi}(1/2+it)dt
=
\frac{\vol(\GmodH)}{4\pi}T^2 + \frac{\mathfrak{C}}{\pi} T \log T + O(T),
\end{equation*}
where we recall that $\mathfrak{C}$ is the number of inequivalent cusps of $\Gamma$.
Consider the above equation at the point $T+1$, and subtract from it the
same quantity for $T$. In order to shorten notation we write $f(t)=-\varphi'/\varphi(1/2+it)$.
We get
\begin{equation}\label{eq:weyl-short}
\sharp\{T\leq t_j\leq T+1\}
+
\frac{1}{4\pi}\int_{-T-1}^{-T}f(t)dt
+
\frac{1}{4\pi}\int_T^{T+1} f(t)dt
=
O(T).
\end{equation}
The function $f(t)$ is bounded from below by a constant $k$ that depends on the group
(this follows from the Maass-Selberg relations, see \cite[eq. (10.9)]{iwaniec_spectral_2002}).
Hence we have
\[\int_{-T-1}^{-T} f(t)dt \geq k.\]
Since the number $\sharp\{T\leq t_j\leq T+1\}$ is non-negative,
we can write
\begin{align*}
\int_T^{T+1} f(t)dt
&\leq
\int_T^{T+1} f(t)dt + \int_{-T-1}^{-T} f(t)dt - k + 4\pi\cdot\sharp\{T\leq t_j\leq T+1\}
\\
&=O(T)-k = O(T).
\end{align*}
Again from the fact that $f(t)$ is bounded from below by a constant, we infer that in fact we have
\[\int_T^{T+1} |f(t)|dt \ll T.\]
Similarly we find
\[\int_{-T-1}^{-T} |f(t)|dt \ll T,\]
and from \eqref{eq:weyl-short} we conclude that we also have $\sharp\{T\leq t_j\leq T+1\}\ll T$.
This proves the lemma.
\end{proof}
%%%% END OF PROOF LEMMA SHORT INTERVAL WEYL INEQUALITY  %%%%
%%%%%%%%%%%%%%%%%%%%%%%%%%%%%%%%%%%%%%%%%%%%%%%%%%%%%%%%%%%%

At this point we can prove bounds on the mean square of the discrete and continuous spectrum
in the Selberg trace formula. We discuss first the discrete spectrum.

%%%%%%%%%%%%%%%%%%%%%%%%%%%%%%%%%%%%%%%%%%%%%%%%%%%%%%%%%%%
%%%%%  PROPOSITION: MEAN SQUARE OF DISCRETE SPECTRUM  %%%%%
\begin{prop}\label{prop01}
Let $0<\delta<1/4$, let $h_\pm$ as in \eqref{2.14A}, and let $R>1$. We have
\begin{equation*}
\int_\RR \Big|\sum_{t_j>0}h_\pm(t_j)\Big|^2 \, w_R(\rho) d\rho 
\ll
\frac{e^R}{\delta} + e^{R/2}\log^2(\delta^{-1}).
\end{equation*}
\end{prop}
%%%%%%%%%%%%%%%%%%%%%%%%%%%%%%%%%%%%%
%%%%%   PROOF OF PROPOSITION   %%%%%%
\begin{proof}
Recall that $h_\pm(t)=h_{\rho\pm\delta(t)}\wh{q}_\delta(t)$, 
and that $\wh{q}_\delta(t)\ll(1+|\delta t|^k)^{-1}$ for every $k\geq 0$.
Using Lemma \ref{lemma006} this implies that the series
\[\sum_{t_j>0} h_\pm(t_j)\]
is absolutely convergent, and we can write
\[
\int_\RR \Big|\sum_{t_j>0} h_\pm(t_j)\Big|^2 w_R(\rho)d\rho 
=
\sum_{t_j>0}\sum_{t_\ell>0} \wh{q}_\delta(t_j)\wh{q}_\delta(t_\ell)
\int_\RR h_{\rho\pm\delta}(t_j)h_{\rho\pm\delta}(t_\ell)w_R(\rho)d\rho. 
\]
By Lemma \ref{lemma005} we can estimate the integral and bound the double sum by
\begin{equation}\label{eq015}
\ll
e^R \sum_{t_j,t_\ell>0} \frac{|\wh{q}_\delta(t_j)\wh{q}_\delta(t_\ell)|v(t_1)v(t_2)}{1+|t_j-t_\ell|^2}
+
e^{R/2}\sum_{t_j,t_\ell>0} |\wh{q}_\delta(t_j)\wh{q}_\delta(t_\ell)|v(t_1^2)v(t_2^2),
\end{equation}
where $v(t)=(1+|t|)^{-1}$.
Using Lemma \ref{lemma001} to bound $\wh{q}_\delta(t)\ll(1+|\delta t|^k)^{-1}$ for every $k\geq 0$,
we can estimate the second sum in \eqref{eq015} by
\begin{equation*}\label{eq016}
\ll e^{R/2}
\Big(
\sum_{t_j\leq \delta^{-1}} \frac{1}{t_j^2}
+
\sum_{t_j>\delta^{-1}} \frac{1}{\delta t_j^3}
\Big)^2
\ll
e^{R/2} (\log^2(\delta^{-1})+1).
\end{equation*}
Consider now the first sum in \eqref{eq015}.
By symmetry and positivity, we can consider only
the sum over $t_\ell\geq t_j$.
Moreover we split the sum in order to optimize the
bounds available.
Consider a unit neighbourhood of the diagonal $t_\ell=t_j$.
Using Lemma \ref{lemma006} we can estimate
\begin{align*}
e^R \sum_{t_j>0}\sum_{t_j\leq t_\ell\leq t_j+1} \frac{|\wh{q}_\delta(t_j)\wh{q}_\delta(t_\ell)|}{(1+|t_j|)(1+|t_\ell|)}
&\ll
e^R \sum_{t_j\leq\delta^{-1}} \frac{1}{t_j}\sum_{t_j\leq t_\ell\leq t_j+1} \frac{1}{t_\ell}
\\
&+
\frac{e^R}{\delta^2} \sum_{t_j>\delta^{-1}} \frac{1}{t_j^2}\sum_{t_j\leq t_\ell\leq t_j+1}\frac{1}{t_\ell^2}
\ll
\frac{e^R}{\delta}.
\end{align*}
The tail of the double sum, that is, the range $t_j>\delta^{-1}$ and $t_\ell>t_j+1$,
can be analyzed (we follow here the same method as in Cram\'er \cite{cramer_mittelwertsatz_1922}) by using
a unit interval decomposition for the sum over $t_\ell$, together with Lemma \ref{lemma006}, to get
\begin{align*}
\frac{e^R}{\delta^2}\!\!\sum_{t_j>\delta^{-1}}\frac{1}{t_j^2}\sum_{t_\ell>t_j+1}\frac{1}{t_\ell^2|t_\ell-t_j|^2}
&\ll
\frac{e^R}{\delta^2}\!\!\sum_{t_j>\delta^{-1}}\frac{1}{t_j^2}\sum_{k=1}^\infty\frac{1}{(t_j+k)k^2}
\ll
\frac{e^R}{\delta^2}\!\!\sum_{t_j>\delta^{-1}}\frac{1}{t_j^3}
\ll
\frac{e^R}{\delta}.
\end{align*}
Finally, the range $t_j\leq \delta^{-1}$ and $t_\ell>t_j+1$ is bounded by
\[
\ll
e^R\sum_{t_j\leq \delta^{-1}} \frac{1}{t_j} \sum_{t_\ell>t_j+1} \frac{1}{t_\ell|t_\ell-t_j|^2}
\ll
e^R\sum_{t_j\leq \delta^{-1}} \frac{1}{t_j} \sum_{k=1}^\infty \frac{1}{k^2}
\ll
\frac{e^R}{\delta}.
\]
We conclude that we have
\[
\int_\RR \Big|\sum_{t_j>0}h_\pm(t_j)\Big|^2 \, w_R(\rho) d\rho 
\ll
\frac{e^R}{\delta} + e^{R/2}\log^2(\delta^{-1}),
\]
as claimed.
\end{proof}
%%%%%   END OF PROOF OF PROPOSITION   %%%%%%
%%%%%%%%%%%%%%%%%%%%%%%%%%%%%%%%%%%%%%%%%%%%

The analysis of the continuous spectrum is similar, and we obtain the same bounds.
With the proposition below we conclude the list of auxiliary results needed to prove Theorem~\ref{intro:thm01}.

%%%%%%%%%%%%%%%%%%%%%%%%%%%%%%%%%%%%%%%%%%%%%%%%%%%%%
%%%%  LEMMA: MEAN SQUARE OF CONTINUOUS SPECTRUM  %%%%
\begin{prop}\label{prop02}
Let $0<\delta<1/4$ and $h_\pm$ be as in \eqref{2.14A}.
Let $R\geq 1$. Then
\begin{equation}\label{eq:030}
\int_\RR \Big|\int_\RR h_\pm(t)\frac{-\varphi'}{\varphi}\big(1/2+it\big)\Big|^2 w_R(\rho) d\rho 
\ll
\frac{e^R}{\delta} + e^{R/2}\log^2(\delta^{-1}).
\end{equation}
\end{prop}
%%%%%%%%%%%%%%%%%%%%%%%%%%%%%
%%%%  PROOF PROPOSITION  %%%%
\begin{proof}
Recall that we have $h_\pm(t)=h_{\rho\pm\delta}(t)\wh{q}_\delta(t)$, and that 
$\wh{q}_\delta(t)\ll (1+|\delta t|^k)^{-1}$ for every $k\geq 0$.
For simplicity we write
\[f(t)=\frac{-\varphi'}{\varphi}\big(1/2+it\big).\]
Let $J$ denote the integral in \eqref{eq:030}.
Due to the decay properties of $\wh{q}_\delta$ and to Lemma~\ref{lemma006},
$J$ is absolutely convergent, and so we can write
\[
J=
\int_\RR\int_\RR f(t_1)\overline{f(t_2)}\wh{q}_\delta(t_1)\wh{q}_\delta(t_2)
\int_\RR h_{\rho\pm\delta}(t_1)h_{\rho\pm\delta}(t_2) w_R(\rho) d\rho\, dt_1 dt_2. 
\]
The innermost integral is bounded using Lemma \ref{lemma005}. This gives
\begin{equation}\label{eq:031}
\begin{aligned}
J
&\ll
e^R \int_\RR \int_\RR \frac{|f(t_1)f(t_2)\wh{q}_\delta(t_2)\wh{q}_\delta(t_2)|}{(1+|t_1|)(1+|t_2|)(1+|t_1-t_2|^2)}dt_1dt_2
\\
&+
e^{R/2} \int_\RR \int_\RR \frac{|f(t_1)f(t_2)\wh{q}_\delta(t_2)\wh{q}_\delta(t_2)|}{(1+t_1^2)(1+t_2^2)}dt_1dt_2.
\end{aligned}
\end{equation}
The second integral in \eqref{eq:031} is bounded by
\[
\begin{aligned}
\ll
\Big(\int_0^{\delta^{-1}} \frac{|f(t)|}{1+t^2}dt + \int_{\delta^{-1}}^\infty \frac{|f(t)|}{\delta t^3}dt\Big)^2
\ll
\big(\log(\delta^{-1})+1\big)^2
\ll
\log^2(\delta^{-1}),
\end{aligned}
\]
where in the first inequality we have used
a unit interval decomposition of the domain of integration,
and Lemma \ref{lemma006} to bound the integral of $|f(t)|$ in unit intervals.
Consider now the first integral in \eqref{eq:031}.
By symmetry and positivity we can consider only the integral over $t_2\geq t_1\geq 0$.
A unit neighbourhood of the diagonal $t_1=t_2$ gives
\[
\begin{aligned}
\int_0^{\delta^{-1}} \frac{|f(t_1)|}{1+t_1} \int_{t_1}^{t_1+1} \frac{|f(t_2)|}{1+t_2}dt_2dt_1
&+
\int_{\delta^{-1}}^\infty \frac{|f(t_1)|}{\delta^2 t_1^3} \int_{t_1}^{t_1+1} \frac{|f(t_2)|}{\delta^2 t_2^3}dt_2dt_1
\\
&\ll
\int_0^{\delta^{-1}} \frac{|f(t_1)|}{1+t_1}dt_1 + \int_{\delta^{-1}}^\infty \frac{|f(t_1)|}{\delta^4 t_1^4}dt_1
\ll
\frac{1}{\delta}.
\end{aligned}
\]
The tail of the double integral, that is, the range $t_1\geq\delta^{-1}$ and $t_2\geq t_1+1$,
can be bounded as follows:
\[
\int_{\delta^{-1}}^\infty \frac{|f(t_1)|}{\delta^2 t_1^3}
\int_{t_1+1}^\infty \frac{|f(t_2)|}{\delta^2 t_2^3 |t_2-t_1|^2} dt_2 dt_1
\ll
\int_{\delta^{-1}}^\infty \frac{|f(t_1)|}{\delta^4 t_1^5}dt_1
\ll
1.
\]
The range $t_1\leq \delta^{-1}$ and $t_2\geq t_1+1$ contributes
\[
\int_0^{\delta^{-1}} \frac{|f(t_1)|}{1+t_1}\int_{t_1+1}^\infty
\frac{|f(t_2)|}{(1+t_2)(1+\delta^2 t_2^2)|t_2-t_1|^2}dt_2 dt_1
\ll
\int_0^{\delta^{-1}} \frac{|f(t_1)|}{1+t_1} dt_1
\ll
\delta^{-1}.
\]
Summarizing, we have showed that we have the bound
\[J \ll \frac{e^R}{\delta} + e^{R/2}\log^2(\delta^{-1}),\]
which is what we wanted. This proves the proposition.
\end{proof}
%%%%% END OF PROOF PROPOSITION CONTINUOUS SPECRTUM  %%%%
%%%%%%%%%%%%%%%%%%%%%%%%%%%%%%%%%%%%%%%%%%%%%%%%%%%%%%%%

%%%%%%%%%%%%%%%%%%%%%%%%%%%%%%%%%%%%%%%%%%%%%%%%%%%%%%%

\section{Modular group}

This section is devoted to the proof of Theorem~\ref{intro:thm02}, concerning the case $G = \text{PSL}_2(\mathbb{Z})$.
Our approach starts with a lemma of Iwaniec \cite[Lemma 1]{iwaniec_prime_1984}, which gives

\begin{equation*}
 \psi_{G}(X) = X + 2\Re\left(\sum_{t_j \leq T} \frac{X^{1/2 + i t_j}}{1/2 + i t_j}\right) + O\left( \frac{X}{T} \log^2 X\right),
\end{equation*}
where $1 \leq T \leq X^{1/2}(\log X)^{-2}$ (and it is understood that the sum runs over $t_j>0$).
Note that in this case the only small eigenvalue of~$\Delta$
is $\lambda=0$, so that $M_{G}(X) = X$ and $P_{G}(X) = \psi_{G}(X)-X$.  From the equation above we deduce that

\begin{equation}
\label{avgerror}
\frac{1}{A} \int_A^{2A} \left|P_{G}(X) \right|^2 dX
\ll
\frac{1}{A} \int_A^{2A} \left| \sum_{t_j \leq T} \frac{X^{1/2 + i t_j}}{1/2 + i t_j} \right|^2 dX
+
\frac{A^2}{T^2} \log^2 A.
\end{equation}

We now describe the outline of the proof of Theorem \ref{intro:thm02}.
The main idea is to use the Kuznetsov trace formula to estimate the mean square of
\begin{equation*}
 \sum_{t_j \leq T} \frac{X^{1/2 + i t_j}}{1/2 + i t_j}
\end{equation*} or, equivalently by partial summation, the mean square of
\begin{equation}
\label{spectralsum}
 R(X,T) = \sum_{t_j \leq T} X^{i t_j}.
\end{equation}
Therefore,
we start by setting up the Kuznetsov trace formula
with a test function that gives us a a smooth version of the sum in \eqref{spectralsum}, and we estimate
the mean square of the weighted sum of Kloosterman sums that show up in its geometric side (Lemma~\ref{avgkloost}).
This allows us to bound the spectral side of the trace formula (Lemma~\ref{kuznetsovbound}).
In order to control the behavior of the Fourier coefficients in the spectral sums of the trace formula
we use a smoothed average of these sums.
This way we obtain (Lemma~\ref{smoothsumbnd}) a mean square estimate of a smoothed version of
$R(X,T)$, from which we extract a mean square estimate for the sharp sum (Lemma~\ref{sharpsumbnd}).

We now start by setting up the Kuznetsov trace formula.
Let $\phi(x)$ be a smooth function on $[0,\infty]$ such that
\begin{align*}
|\phi(x)| &\ll x, \quad x \rightarrow 0, \\
|\phi^{(l)}(x)| &\ll x^{-3}, \quad x \rightarrow \infty,
\end{align*}
for $l=0,1,2,3$. Define
\begin{align*}
&\phi_0  = \frac{1}{2\pi} \int_0^{\infty} J_0(y) \phi(y) dy,\\
&\phi_B(x)  = \int_0^1 \int_0^{\infty} \xi x J_0(\xi x) J_0(\xi y) \phi(y) dy d\xi, \\
&\phi_H(x)  = \int_1^{\infty} \int_0^{\infty} \xi x J_0(\xi x) J_0(\xi y) \phi(y) dy d\xi,\\
&\hat{\phi}(t) = \frac{\pi}{2i \sinh \pi t} \int_0^{\infty} \left( J_{2it}(x) - J_{-2it}(x) \right) \phi(x) \frac{dx}{x},
\end{align*}
where $J_\nu$ is the Bessel function of the first kind and order $\nu$.
By the properties of the Hankel transform we have
\begin{equation*}
 \phi(x) = \phi_B(x) + \phi_H(x).
\end{equation*}
We now choose $\phi$ as in Luo-Sarnak \cite{luo_quantum_1995}. For $X,T>1$ we set
\begin{align*}
\phi_{X,T}(x) & = \frac{-\sinh \beta}{\pi} x \exp(i x \cosh \beta), \\
2\beta & = \log X + \frac{i}{T},
\end{align*}
and apply the Kuznetsov trace formula \cite[Th. 1]{kuznetsov_petersson_1980} with $\phi_{X,T}$ as the test function.
Let $\{f_i\}_{i=1}^{\infty}$ be an orthonormal basis of Maass cusp forms for $\text{SL}_2(\mathbb{Z})$,
with eigenvalues $\lambda_j = \frac{1}{4} + t_j^2$. These cusp forms have Fourier expansions \cite[eq. (2.10)]{kuznetsov_petersson_1980}
\begin{equation*}
f_j(z) = \sqrt{y} \sum \limits_{n=1}^{\infty} \rho_j(n) K_{i t_j}(2\pi n y) \cos(2\pi n x),
\end{equation*}
where $K_\nu(x)$ is the $K$-Bessel function.
For $l_1, l_2 \geq 1$ the trace formula reads
\begin{align*}
\sum_{t_j} \hat{\phi}(t_j) \nu_j(l_1) \overline{\nu_j(l_2)} &+ \frac{2}{\pi} \int_0^{\infty} \frac{\hat{\phi}(t)}{|\zeta(1+2it)|^2}
d_{it}(l_1) d_{it}(l_2) dt \\
&= \delta_{l_1, l_2} \phi_0 + \sum \limits_{c=1}^{\infty} \frac{S(l_1, l_2, c)}{c} \phi_H \left( \frac{4\pi\sqrt{l_1 l_2}}{c}\right),
\end{align*} 
where $\rho_j(l) = \nu_j(l) \cosh(\pi t_j)^{1/2}$,
$d_{it}(l)=\sum_{d_1d_2=l} (d_1/d_2)^{it}$,
and $S(l_1, l_2, c)$ is the classical Kloostermann sum.
By \cite[p. 234]{luo_quantum_1995} we have
\begin{equation}
\label{firstbnd}
\hat{\phi}_{X,T}(t_j) = X^{i t_j} e^{-t_j/T} + O(e^{-\pi t_j}),
\end{equation}
\begin{equation}
\label{secondbnd}
(\phi_{X,T})_0 \ll X^{-1/2},
\end{equation}
\begin{equation}
\frac{2}{\pi} \int_0^{\infty} \frac{\hat{\phi}_{X,T}(t)}{|\zeta(1+2it)|^2}
(d_{it}(n))^2 dt \ll T \log^2 T \,d^2(n),
\end{equation}
\begin{equation} 
\label{lastbnd}
S_n((\phi_{X,T})_B) \ll n^{1/2} X^{-1/2} \log^2 n,
\end{equation}
where
\begin{equation}\label{2701:eq004}
S_n(\psi) = \sum \limits_{c=1}^{\infty} \frac{S(n,n,c)}{c} \psi\left(\frac{4\pi n}{c}\right).
\end{equation}

Analyzing the right hand side of the Kuznetsov trace formula, we prove a bound for the mean square of $S_n(\phi_{X,T})$ as follows.
\begin{lemma}
\label{avgkloost}
Let $A, T > 2$ and let $n$ be a positive integer. Then, for any $\varepsilon > 0$,
\begin{equation*}
\frac{1}{A} \int_A^{2A} \left| S_n(\phi_{X,T}) \right|^2 dX \ll_{\varepsilon} (n A^{1/2} + T^2) (An)^{\varepsilon}.
\end{equation*}
\end{lemma}
We postpone the proof of Lemma \ref{avgkloost} to the end of the section,
and we show here how to recover a similar bound on the spectral side of the Kuznetsov trace formula.
\begin{lemma}
\label{kuznetsovbound}
Let $A, T > 2$ and let $n$ be a positive integer. Then, for any $\varepsilon > 0$,
\begin{equation*}
\frac{1}{A} \int_A^{2A} \left| \sum_{t_j} |\nu_j(n)|^2 \hat{\phi}(t_j) \right|^2 dX \ll_{\varepsilon} (n A^{1/2} + T^2) (AnT)^{\varepsilon}.
\end{equation*}
\end{lemma}

\begin{proof}
By the trace formula and the bounds in equations~\eqref{secondbnd}--\eqref{lastbnd}
we deduce
\begin{equation*}
\begin{split}
\frac{1}{A} \int_A^{2A} \left| \sum_{t_j} |\nu_j(n)|^2 \hat{\phi}(t_j) \right|^2 dX
&\ll
\frac{1}{A} \int_A^{2A} \left| S_n(\phi_{X,T}) \right|^2 dX
\\
&+
T^2\log^4 T \, d^4(n) + \frac{n \log A \log^4 n }{A}.
\end{split}
\end{equation*}
Observing that $d^4(n)\ll n^\eps$, the claim follows from Lemma~\ref{avgkloost}.
\end{proof}

Once we have bounds for a fixed~$n$, we average over $n\in[N,2N]$.
Let $h(\xi)$ be a smooth function supported in $[N,2N]$, whose derivatives satisfy
\begin{equation*}
|h^{(p)}(\xi)| \ll N^{-p}, \quad p=0,1,2,\ldots
\end{equation*}
and such that
\begin{equation*}
\int \limits_{-\infty}^{+\infty} h(\xi) d \xi = N.
\end{equation*}
The next lemma is a result of Luo-Sarnak \cite{luo_quantum_1995}.
\begin{lemma}
Let $h$ be as above. Then
\begin{align*}
&\sum_n h(n) |\nu_j(n)|^2 = \frac{12}{\pi^2} N + r(t_j, N),\\
&\sum_{t_j \leq T} |r(t_j, N)| \ll T^2 N^{1/2} \log^2 T.
\end{align*}
\end{lemma}

\begin{proof} See \cite[pg.233]{luo_quantum_1995}.
\end{proof}

\begin{lemma}
\label{smoothsumbnd}
Let $A, T > 2$. Then, for any $\varepsilon > 0$,
\begin{equation*}
\frac{1}{A} \int_A^{2A} \left| \sum_{t_j} X^{it_j} e^{-t_j/T} \right|^2 dX \ll_{\varepsilon} A^{1/4} \, T^2 (A T)^{\varepsilon}.
\end{equation*}
\end{lemma}

\begin{proof}
By the previous lemma
\begin{equation}\label{2701:eq002}
\begin{aligned}
\frac{1}{N} \sum_n h(n) \left(\sum_{t_j} |\nu_j(n)|^2 \hat{\phi}(t_j) \right) & = \frac{1}{N} \sum_{t_j} \left( \sum_n |\nu_j(n)|^2 h(n) \right) \hat{\phi}(t_j) \\
& = \frac{12}{\pi^2} \sum_{t_j} \hat{\phi}(t_j) + \frac{1}{N} \sum_{t_j} r(t_j, N) \hat{\phi}(t_j) \\
& = \frac{12}{\pi^2} \sum_{t_j} \hat{\phi}(t_j) + O(T^2 N^{-1/2} \log^2 T),
\end{aligned}
\end{equation}
where the last step can be obtained by $|\hat{\phi}(t_j)| \ll e^{-t_j/T}$ and partial summation.
Note now that from equation~\eqref{firstbnd} it follows that
\begin{equation}\label{2701:eq001}
\sum_{t_j} \hat{\phi}(t_j) = \sum_{t_j} X^{i t_j} e^{-t_j/T} + O(1).
\end{equation}
Combining \eqref{2701:eq002} and \eqref{2701:eq001} we deduce that
\begin{equation*}
\sum_{t_j} X^{i t_j} e^{-t_j/T} \ll \frac{1}{N} \sum_n h(n) \left(\sum_{t_j} |\nu_j(n)|^2 \hat{\phi}(t_j) \right) + T^2 N^{-1/2} \log^2 T.
\end{equation*}
Taking absolute value, squaring, and integrating over $X$,
we infer that
\begin{equation*}
\begin{split}
\frac{1}{A} \int \limits_A^{2A} \;
&\bigg| \sum_{t_j} X^{it_j} e^{-t_j/T} \bigg|^2 dX
\\
&\ll
\frac{1}{A} \int_A^{2A} \left| \frac{1}{N} \sum_n h(n) \left(\sum_{t_j} |\nu_j(n)|^2 \hat{\phi}(t_j) \right) \right|^2 dX
+
T^4 N^{-1} \log^4 T.
\end{split}
\end{equation*}
Moreover, by Cauchy-Schwarz and the properties of $h$, we have
\begin{align*}
\left| \frac{1}{N} \sum_n h(n) \left(\sum_{t_j} |\nu_j(n)|^2 \hat{\phi}(t_j) \right) \right|^2
&\ll
\frac{1}{N^2} \left( \sum_{n=N}^{2N} |h(n)|^2 \right) \left( \sum_{n=N}^{2N} \left| \sum_{t_j} |\nu_j(n)|^2 \hat{\phi}(t_j) \right|^2\right)
\\
&\ll
\frac{1}{N}  \sum_{n=N}^{2N} \left| \sum_{t_j} |\nu_j(n)|^2 \hat{\phi}(t_j) \right|^2
\end{align*}
Finally, we obtain
\begin{align*}
\frac{1}{A} \int_A^{2A} \left| \sum_{t_j} X^{it_j} e^{-t_j/T} \right|^2 dX
&\ll_{\varepsilon}
\frac{1}{N} \sum_{n=N}^{2N}  \left( \frac{1}{A} \int_A^{2A} \left| \sum_{t_j} |\nu_j(n)|^2 \hat{\phi}(t_j) \right|^2 dX \right) + \frac{T^{4+\varepsilon}}{N}
\\
&\ll_{\varepsilon}
(N A^{1/2} +T^2) (ANT)^{\varepsilon} + \frac{T^{4+\varepsilon}}{N},
\end{align*}
where the last inequality follows from Lemma~\ref{kuznetsovbound}.
Choose $N = A^{-1/4} T^2$ to complete the proof.
\end{proof}

The next step is to replace the smoothed sum
$\sum X^{it_j} e^{-t_j/T}$
with the truncated one.
This gives us a corresponding mean square estimate for $R(X,T)$.

\begin{prop}
\label{sharpsumbnd}
Let $A, T > 2$, and let $R(X,T)$ be as in \eqref{spectralsum}. Then, for any $\varepsilon > 0$,
\begin{equation*}
\frac{1}{A} \int_A^{2A} |R(X,T)|^2 dX \ll_{\varepsilon} A^{1/4}\, T^2 (AT)^{\varepsilon}.
\end{equation*}
\end{prop}

\begin{proof}
We start by choosing a smooth function $g$ that approximates the characteristic function of $[1,T]$.
Let $g$ be a smooth function supported on $[1/2,T+1/2]$ such that $0 \leq g(\xi) \leq 1$, and $g(\xi)=1$ when $\xi \in [1,T]$.
By the strong Weyl's law
we know that $|\{t_j : U \leq t_j \leq U+1 \}| \ll U$, and so we have
\begin{equation*}
R(X,T)
= \sum_{t_j} g(t_j) X^{i t_j} + O(T).
\end{equation*}
Define $\hat{g}(\xi)$ to be the Fourier transform of $g(\xi) \exp(\xi/T)$. By \cite[p. 235-236]{luo_quantum_1995} we have that
\begin{multline*}
\sum_{t_j} g(t_j) X^{i t_j} = \int_{-1}^1 \hat{g}(\xi) \left( \sum_{t_j} (X e^{-2\pi \xi})^{i t_j} e^{-t_j/T} \right) d\xi \\
+ O\left( \sum_{t_j} \frac{e^{-t_j/T}}{t_j} + \frac{e^{-t_j/T}}{|T-t_j|+1} + \frac{\log(T+t_j)e^{-t_j/T}}{T}  \right),
\end{multline*}
where the error term can be bounded by $O(T \log T)$.
Also note the estimate
\begin{equation}\label{2701:eq003}
\hat{g}(\xi) \ll \min \left( T, \frac{1}{|x|}\right).
\end{equation}
Defining
\[
k(X,T,\xi) := \sum \limits_{t_j} (X e^{-2\pi \xi})^{i t_j} e^{-t_j/T},
\]
we deduce that
\begin{align*}
\frac{1}{A} \int_A^{2A} |R(X,T)|^2 dX
&\ll
\frac{1}{A} \int_A^{2A} \left| \int_{-1}^1 \hat{g}(\xi) k(X,T,\xi) d\xi \right|^2 dX + T^2 \log^2 T
\\
&\ll
\frac{1}{A} \int_A^{2A} \left| \int_{|\xi| \leq \delta} \hat{g}(\xi) k(X,T,\xi) d\xi \right|^2 dX
\\
&+
\frac{1}{A} \int_A^{2A} \left| \int_{\delta < |\xi| \leq 1} \hat{g}(\xi) k(X,T,\xi) d\xi \right|^2 dX + T^2 \log^2 T.
\end{align*}
The first term in the sum above we bound, using Lemma~\ref{smoothsumbnd} and the first bound in \eqref{2701:eq003}, by
\begin{equation*}
\int \limits_{|\xi| \leq \delta} |\hat{g}(\xi)|^2  \int \limits_{|\xi| \leq \delta} \frac{1}{A} \int_A^{2A} \left| k(X,T,\xi) \right|^2 dX d\xi
\ll_{\varepsilon}
A^{1/4} \, T^4 (AT)^{\varepsilon} \delta^2,
\end{equation*}
and the second term, using again Lemma \ref{smoothsumbnd} and the second bound in \eqref{2701:eq003}, by
\begin{equation*}
\int \limits_{\delta < |\xi| \leq 1} |\hat{g}(\xi)|^2 |\xi|  \int \limits_{\delta < |\xi| \leq 1} \frac{1}{|\xi|}\frac{1}{A} \int_A^{2A} \left| k(X,T,\xi) \right|^2 dX d\xi  \ll_{\varepsilon} A^{1/4} \, T^2 (AT)^{\varepsilon} \log(\delta^{-1}).
\end{equation*}
On taking $\delta = T^{-2}$ we obtain the claim.
\end{proof}

We are now able to conclude the proof of Theorem~\ref{intro:thm02}.
Let $2 < T \leq A^{1/2}(\log A)^{-2}$. By partial summation,
\begin{equation*}
\sum_{t_j \leq T} \frac{X^{1/2 + i t_j}}{1/2 + i t_j} = \frac{R(X,T)X^{1/2}}{1/2+i T} + i X^{1/2} \int_1^T \frac{R(X,U)}{(1/2+iU)^2} dU.
\end{equation*}
Therefore,
\begin{align*}
\frac{1}{A}\int_A^{2A} \left|\sum_{t_j \leq T} \frac{X^{1/2 + i t_j}}{1/2 + i t_j}\right|^2 dX & \ll \frac{1}{A}\int_A^{2A} \left|\frac{R(X,T)X^{1/2}}{1/2+i T}\right|^2 dX \\
& + \frac{1}{A}\int_A^{2A} \left|X^{1/2} \int_1^T \frac{R(X,U)}{(1/2+iU)^2} dU\right|^2 dX.
\end{align*}
The first term on the right hand side is bounded, in view of Proposition \ref{sharpsumbnd}, by $O(A^{5/4+\varepsilon})$. The second term can be bounded
by using Cauchy-Schwarz inequality and again Proposition \ref{sharpsumbnd}, giving
\[
A (\log T) \int_1^T \frac{1}{U^3} \left( \frac{1}{A}\int_A^{2A} |R(X,U)|^2 dX \right) dU
\ll_{\varepsilon}
A^{5/4+\varepsilon} \int_1^T U^{-1} dU
\ll_{\varepsilon}
A^{5/4+\varepsilon}.
\]
Inserting these bounds in \eqref{avgerror} we obtain
\begin{equation*}
\frac{1}{A} \int_A^{2A} \left| \psi_{G}(X) - X \right|^2 dX \ll_{\varepsilon} A^{5/4+\varepsilon} + \frac{A^2}{T^2} \log^2 A.
\end{equation*}
Choosing $T = A^{3/8}$ above concludes the proof of Theorem~\ref{intro:thm02}.

\subsection{Weighted sums of Kloosterman sums}
In the remainder of the section we prove Lemma~\ref{avgkloost}.
Recall that we want to estimate, on average, the following sum of Kloostermann sums
\begin{equation*}
\frac{1}{A} \int_A^{2A} \left| S_n(\phi_{X,T}) \right|^2 dX,
\end{equation*}
where $S_n(\psi)$ is the weighted sum of Kloosterman sums defined in \eqref{2701:eq004}, i.e.
\begin{equation*}
S_n(\psi) = \sum \limits_{c=1}^{\infty} \frac{S(n,n,c)}{c} \psi\left(\frac{4\pi n}{c}\right),
\end{equation*}
with $n$ a (large) positive integer.
The function $\phi_{X,T}$ carries some oscillation in $X$, that we exploit
when integrating over $X\in[A,2A]$.
\begin{lemma}
\label{integralbounds}
Let $A, T > 2$, and let $z_1, z_2$ be positive real numbers. Then
\begin{equation*}
\frac{1}{A} \int_A^{2A} \phi_{X,T}\left(z_1\right) \overline{\phi_{X,T}\left(z_2\right)} dX \ll z_1 z_2 A \exp\left(-\frac{A^{1/2}(z_1 + z_2)}{T} \right)
\end{equation*}
and
\begin{equation*}
\frac{1}{A} \int_A^{2A} \phi_{X,T}\left(z_1\right) \overline{\phi_{X,T}\left(z_2\right)} dX \ll \frac{z_1 z_2 A^{1/2}}{|z_1 - z_2|}.
\end{equation*}
For the last bound we assume $z_1 \neq z_2$.
\end{lemma}

\begin{proof}
Inserting the definition of $\phi_{X,T}$ we can write the integral as
\begin{equation}
\label{integral}
\frac{z_1 z_2}{\pi^2} \int_A^{2A} \left| \sinh \beta \right|^2 \exp\left(i \cosh \beta z_1 + \overline{i \cosh \beta z_2}\right) dX.
\end{equation}
Bounding the integrand uniformly using
\begin{align*}
\left| \sinh \beta \right|^2 &\ll X, \\
\exp\left(i \cosh \beta z_1 + \overline{i \cosh \beta z_2}\right)& \ll \exp\left(-\frac{A^{1/2}(z_1 + z_2)}{T} \right),
\end{align*}
we obtain the first bound
\begin{equation*}
\frac{1}{A} \int_A^{2A} \phi_{X,T}\left(z_1\right) \overline{\phi_{X,T}\left(z_2\right)} dX \ll z_1 z_2 A \exp\left(-\frac{A^{1/2}(z_1 + z_2)}{T} \right).
\end{equation*}
To obtain the second bound we use integration by parts in equation~\eqref{integral} to get
\[
\begin{split}
\frac{z_1 z_2}{\pi^2} f(X)
&\exp\left(
i \cosh \beta z_1 + \overline{i \cosh \beta z_2}
\right) \bigg|_A^{2A}
\\
&+
\frac{z_1 z_2}{\pi^2} \int \limits_A^{2A} f'(X) \exp\left(i \cosh \beta z_1 + \overline{i \cosh \beta z_2} \right) dX,
\end{split}
\]
where
\begin{equation*}
f(X) = \frac{2X \left| \sinh \beta \right|^2}{i \sinh \beta z_1 + \overline{i \sinh \beta z_2}}.
\end{equation*}
We can bound these terms (up to a constant) by
\begin{equation*}
\frac{z_1 z_2 A^{3/2}}{\left|z_1 - z_2 \right|},
\end{equation*}
which gives us the second bound
\begin{equation*}
 \frac{1}{A} \int_A^{2A} \phi_{X,T}\left(z_1\right) \overline{\phi_{X,T}\left(z_2\right)} dX \ll \frac{z_1 z_2 A^{1/2}}{\left|z_1 - z_2 \right|} .
\end{equation*}
This proves the lemma.
\end{proof}

The lemma above will provide the necessary bounds to establish Lemma~\ref{avgkloost}.
\begin{proof}[Proof of Lemma~\ref{avgkloost}]
We expand the square in the integrand and exchange the integral with the sums,
so that we can rewrite the integral as
\begin{equation}
\label{ksum}
\sum \limits_{c_1, c_2 = 1}^{\infty} \frac{S(n,n,c_1) \overline{S(n,n,c_2)}}{c_1 c_2} \frac{1}{A} \int_A^{2A} \phi_{X,T}\left(\frac{4 \pi n}{c_1}\right) \overline{\phi_{X,T}\left(\frac{4 \pi n}{c_2}\right)} dX.
\end{equation}
We now split the sum in equation~\eqref{ksum} into two sums $\Sigma_d$ and $\Sigma_{nd}$,
where $\Sigma_d$ is the sum over the diagonal terms $c_1 = c_2$, and $\Sigma_{nd}$ is the sum over the terms $c_1 \neq c_2$.
We shall make use of the Weil bound on Kloosterman sums throughout the proof, namely
\begin{equation}\label{2701:eq005}
|S(n,n,c)| \leq (n,c)^{1/2} c^{1/2} d(c).
\end{equation}
Moreover, we have
\begin{equation*}
\sum_{c \leq x} (n,c) \, d^2(c) \ll x \log^3 x \, d(n).
\end{equation*}
We bound the diagonal terms using the first bound in Lemma~\ref{integralbounds} (with $z_1=z_2=(4\pi n)/c$), obtaining
\begin{equation*}
\Sigma_d \ll n^2 A \sum \limits_{c} \frac{(n,c) d^2(c)}{c^3} \exp\left( - \frac{A^{1/2} n}{T c} \right)
\ll T^2 \log^3 (An) \, d(n).
\end{equation*}
To bound the non-diagonal terms we interpolate the two bounds in Lemma~\ref{integralbounds}
to get, for $0<\lambda<1$,
\begin{equation*}
\frac{1}{A} \int_A^{2A} \phi_{X,T}\left(\frac{4 \pi n}{c_1}\right) \overline{\phi_{X,T}\left(\frac{4 \pi n}{c_2}\right)} dX \ll \left( \frac{n^2 A}{c_1 c_2} \right)^{1-\lambda} \left( \frac{n A^{1/2}}{|c_1 -c_2|} \right)^{\lambda}.
\end{equation*}
Therefore,
\begin{align*}
\Sigma_{nd} & \ll \sum \limits_{c_1 \neq c_2 = 1}^{\infty} \frac{|S(n,n,c_1) S(n,n,c_2)|}{c_1 c_2} \left( \frac{n^2 A}{c_1 c_2} \right)^{1-\lambda} \left( \frac{n A^{1/2}}{|c_1 -c_2|} \right)^{\lambda} \\
& \ll_{\lambda} (n A^{1/2})^{2-\lambda} \left(\sum_{c=1}^{\infty} \frac{|S(n,n,c)|^{\frac{2}{2-\lambda}}}{c^2} \right)^{2-\lambda},
\end{align*}
where the last inequality follows from the Hardy-Littlewood-P\'olya inequality \eqref{HLP}.
Applying the Weil bound \eqref{2701:eq005} we obtain
\begin{align*}
\Sigma_{nd} & \ll_{\lambda} (n A^{1/2})^{2-\lambda} \left(\sum_{c=1}^{\infty} \frac{((n,c)^{1/2} d(c) \sqrt{c}) ^{\frac{2}{2-\lambda}}}{c^2} \right)^{2-\lambda} \\
& \ll_{\lambda} (n A^{1/2})^{2-\lambda} \left(\sum_{c=1}^{\infty} \frac{(n,c) d^2(c)}{c^{1+\frac{1-\lambda}{2-\lambda}}} \right)^{2-\lambda}
 \ll_{\lambda} (n A^{1/2} d(n))^{2-\lambda}.
\end{align*}
Pick $\lambda = 1-\varepsilon$ and note that $d(n) \ll_{\varepsilon} n^{\epsilon}$ to finish the proof.
\end{proof}

%%%%%%%%%%%%%%%%%%%%%%%%%%%%%%%%%%%%%%%%%%%%%%%%%%%%%%


\begin{thebibliography}{10}


\bibitem{avdispahic_koyamas_2017}
M. Avdispahi\'c, \emph{On {Koyama}'s refinement of the prime geodesic theorem},
   preprint, \url{https://arxiv.org/abs/1701.01642} (2017).

\bibitem{avdispahic_prime_2017}
M. Avdispahi\'c, \emph{Prime geodesic theorem of {Gallagher} type},
   preprint, \url{https://arxiv.org/abs/1701.02115} (2017).

\bibitem{cai_prime_2002}
Y. Cai, \emph{Prime geodesic theorem},
   J. Th\'eor. Nombres Bordeaux \textbf{14} (2002), no.~1, 59--72.

\bibitem{carneiro_extremal_2010}
E. Carneiro and J. Vaaler, \emph{Some extremal functions in {Fourier} analysis. {II}},
   Trans. Amer. Math. Soc. \textbf{362} (2010), no.~11, 5803--5843.

\bibitem{cramer_mittelwertsatz_1922}
H. Cram\'er, \emph{Ein {Mittelwertsatz} in der {Primzahltheorie}},
  Math. Z. \textbf{12} (1922), no.~1, 147--153.

\bibitem{gallagher_large_1970}
P.~X. Gallagher, \emph{A large sieve density estimate near $\sigma=1$},
   Invent. Math. \textbf{11} (1970), no.~4, 329--339.

\bibitem{hardy_inequalities_1934}
G. Hardy, J. E. Littlewood, G. P\'{o}lya, \emph{Inequalities}, Cambridge Univ. Press, 1934.

\bibitem{hejhal_selberg_1983}
D.~A. Hejhal, \emph{The {Selberg} trace formula for $\mathrm{PSL}(2,\mathbf{R})$. {Vol}. 2},
  Lecture {Notes} in {Mathematics} \textbf{1001}, Springer-Verlag, Berlin, 1983.

\bibitem{huber_zur_1961}
H. Huber, \emph{Zur analytischen {Theorie} hyperbolischer {Raumformen} und {Bewegungsgruppen}. {II}},
   Math. Ann. \textbf{142} (1961), no.~4, 385--398.

\bibitem{huber_zur_1961-1}
H. Huber, \emph{Zur analytischen {Theorie} hyperbolischer {Raumformen} und {Bewegungsgruppen}. {II}, Nachtrag},
   Math. Ann. \textbf{143} (1961), no.~5, 463--464.

\bibitem{ivic_theory_1985}
A. Ivi\'c, \emph{The Riemann zeta-function. The theory of the Riemann zeta-function with applications},
   J. Wiley \& Sons, Inc., New York, 1985.

\bibitem{iwaniec_nonholomorphic_1984}
H. Iwaniec, \emph{Nonholomorphic modular forms and their applications},
   Modular forms ({Durham}, 1983),
   Ellis {Horwood} {Ser}. {Math}. {Appl}.: {Statist}. {Oper}. {Res}.,
   Horwood, Chichester, 1984, pp.~157--196.

\bibitem{iwaniec_prime_1984}
H.~Iwaniec, \emph{Prime geodesic theorem.},
   J. Reine Angew. Math. \textbf{1984} (1984), no.~349, 136--159.

\bibitem{iwaniec_spectral_2002}
H. Iwaniec, \emph{Spectral methods of automorphic forms},
   Second Ed., Graduate {Studies} in {Mathematics},
   \textbf{53}, Amer. Math. Soc., Providence, RI;
   Rev. Mat. Iberoamer., Madrid, 2002.

\bibitem{koyama_prime_1998}
S.-Y. Koyama, \emph{Prime geodesic theorem for arithmetic compact surfaces},
  Internat. Math. Res. Notices \textbf{1998} (1998), no.~8, 383--388.

\bibitem{koyama_refinement_2016}
S.-Y. Koyama, \emph{Refinement of prime geodesic theorem},
   Proc. Japan Acad., Ser. A, Math. Sci. \textbf{92} (2016), no.~7, 77--81.
  
\bibitem{kuznetsov_petersson_1980}
N. V. Kuznetsov, \emph{The {Petersson} conjecture for cusp forms of weight zero and the Linnik conjecture. Sums of Kloosterman sums},
   Mat. Sb. (N. S.), \textbf{111(153)}, 334-383, 1980 (Russian).

\bibitem{li_extension_2011}
C. Li and J. Villavert, \emph{An extension of the Hardy-Littlewood-P\'olya inequality},
   Acta Math. Sci. Ser. B Engl. Ed. \textbf{31} (2011), no.~6, 2285--2288.

\bibitem{luo_quantum_1995}
W. Luo and P. Sarnak, \emph{Quantum ergodicity of {Eigenfunctions} on $\mathrm{PSL}_2(\mathbf{Z})\backslash\mathbf{H}^2$},
   Inst. Hautes \'Etudes Sci. Publ. Math. \textbf{81} (1995), 207--237.

\bibitem{luo_selbergs_1995}
W.~Luo, Z.~Rudnick, and P.~Sarnak, \emph{On {Selberg}'s eigenvalue conjecture},
   Geom. Funct. Anal. \textbf{5} (1995), no.~2, 387--401.

\bibitem{petridis_local_2014}
Y.~N. Petridis and M.~S. Risager, \emph{Local average in hyperbolic lattice point counting, with an Appendix by Niko Laaksonen},
   Math. Z. (2016). \textsf{doi: 10.1007/s00209-016-1749-z}

\bibitem{phillips_circle_1994}
R.~Phillips and Z.~Rudnick, \emph{The {Circle} {Problem} in the {Hyperbolic} {Plane}},
   J. Funct. Anal. \textbf{121} (1994), no.~1, 78--116.

\bibitem{rubinstein_chebyshevs_1994}
M. Rubinstein and P. Sarnak, \emph{Chebyshev's bias},
   Experiment. Math. \textbf{3} (1994), no.~3, 173--197.

\bibitem{sarnak_prime_1980}
P.~C. Sarnak, \emph{Prime {Geodesic} {Theorems}},
   Thesis (Ph.D.), Stanford University, 1980, 111 pp., ProQuest LLC.

\bibitem{sarnak_class_1982}
P. Sarnak, \emph{Class numbers of indefinite binary quadratic forms},
  J. Number Theory \textbf{15} (1982), no.~2, 229--247.

\bibitem{soundararajan_prime_2013}
K.~Soundararajan and M.~P. Young, \emph{The prime geodesic theorem},
  J. Reine Angew. Math. \textbf{2013} (2013), no.~676 105--120.

\bibitem{venkov_spectral_1990}
A.~B. Venkov, \emph{Spectral theory of automorphic functions and its applications},
   Mathematics and its {Applications} ({Soviet} {Series}), \textbf{51},
   Kluwer Academic Publishers Group, Dordrecht, 1990, Translated from the Russian by N. B. Lebedinskaya.

\bibitem{wintner_asymptotic_1935}
A. Wintner, \emph{On the {Asymptotic} {Distribution} of the {Remainder} {Term} of the {Prime}-{Number} {Theorem}},
   Amer. J. Math. \textbf{57} (1935), no.~3, 534--538.



\end{thebibliography}
\end{document}